\documentclass[onefignum,onetabnum]{siamart190516}
\usepackage{lipsum}
\usepackage{amsfonts}
\usepackage{graphicx}
\usepackage{epstopdf}
\usepackage{algorithmicx}
\usepackage{algpseudocode}
\ifpdf
  \DeclareGraphicsExtensions{.eps,.pdf,.png,.jpg}
\else
  \DeclareGraphicsExtensions{.eps}
\fi

\graphicspath{{./}}

\usepackage{dsfont}
\usepackage{mathtools}
\usepackage{enumitem}
\usepackage{booktabs}
\usepackage{subcaption}
\usepackage{placeins}	% for \FloatBarrier command

\newcommand{\C}{\mathbb{C}}
\newcommand{\R}{\mathbb{R}}

\newcommand{\W}{\mathbb{W}}
\newcommand{\one}{\boldsymbol{1}}
\newcommand{\zero}{\boldsymbol{0}}

\newcommand{\rat}{\mathcal{Q}}

\newcommand{\barbarvec}[1]{\bar{\bar {\vec{#1}}}} % barbar+vec
\newcommand{\barvec}[1]{\bar {\vec{#1}}}	% bar+vec
\DeclareMathOperator{\vspan}{span}
\DeclareMathOperator{\range}{range}
\DeclarePairedDelimiter{\norm}{\lVert}{\rVert}	
\DeclarePairedDelimiter{\abs}{\lvert}{\rvert}	

\renewcommand{\vec}{\boldsymbol}

\makeatletter
\renewcommand*\env@matrix[1][*\c@MaxMatrixCols c]{%
  \hskip -\arraycolsep
  \let\@ifnextchar\new@ifnextchar
  \array{#1}}
\makeatother

\newsiamremark{remark}{Remark}
\newsiamremark{hypothesis}{Hypothesis}
\crefname{hypothesis}{Hypothesis}{Hypotheses}
\newsiamthm{claim}{Claim}

\headers{Rational Krylov for fractional diffusion on graphs}{M. Benzi and I. Simunec}

\title{Rational Krylov methods for fractional diffusion problems on graphs\thanks{Submitted to the editors DATE.
\funding{ -- }}}

\author{Michele Benzi\thanks{Scuola Normale Superiore, Piazza dei Cavalieri 7, 56126 Pisa, Italy 
  (\email{michele.benzi@sns.it}, \email{igor.simunec@sns.it}).}
\and Igor Simunec\footnotemark[2]}

\usepackage{amsopn}
\DeclareMathOperator{\diag}{diag}

\ifpdf
\hypersetup{
  pdftitle={Rational Krylov methods for fractional diffusion problems on graphs},
  pdfauthor={}
}
\fi

\begin{document}

\maketitle

% REQUIRED
\begin{abstract}
	In this paper we propose a method to compute the solution to the fractional diffusion equation on directed networks, which can be expressed in terms of the graph Laplacian $L$ as a product $f(L^T) \boldsymbol{b}$, where $f$ is a non-analytic function involving fractional powers and $\boldsymbol{b}$ is a given vector. The graph Laplacian is a singular matrix, causing Krylov methods for $f(L^T) \boldsymbol{b}$ to converge more slowly. In order to overcome this difficulty and achieve faster convergence, we use rational Krylov methods applied to a desingularized version of the graph Laplacian, obtained with either a rank-one shift or a projection on a subspace.
\end{abstract}

% REQUIRED

% REQUIRED

\section{Introduction}
The use of graph models to represent complex structures is extremely widespread, ranging from real and digital social networks, to transportation networks, to networks of chemical reactions and many others. It is often of interest in applications to study diffusion processes taking place on a network, that evolve in time according to the distribution of edges in the underlying graph. Such a process can be described in terms of a system of ordinary differential equations in time, with the Laplacian matrix $L$ of the graph as the coefficient matrix, and its solution at time $t$ can be written as $\vec u(t) = \exp(-t L^T) \vec u(0)$. An expression of this form can be computed efficiently using a polynomial Krylov method.

% add some remark on directed vs. undirected graphs?

While a diffusion process is a local phenomenon, there are certain phenomena that allow long-range interactions and are non-local in nature: in the continuous setting, phenomena of this kind have been effectively modelled using fractional powers of the Laplace differential operator, that is $(-\Delta)^{\alpha}$ for $\alpha \in {[1/2, 1)}$ (see, e.g., \cite[Definition~2]{ILTA06}). Analogously, in the context of directed graphs, these phenomena can be well described with a fractional diffusion model, which employs a fractional power of the graph Laplacian instead of the Laplacian itself. Unlike for the continuous Laplace operator, in the discrete case there is no need to restrict the values of $\alpha$ to the interval ${[1/2, 1)}$, and we can use any exponent $\alpha \in (0,1)$. This modelling approach has been recently studied in \cite{MateosRiascos14} in the undirected case, and in \cite{BBDS20} for more general directed graphs; we refer to the book \cite{MRCNN19} for more information on fractional dynamics. We mention that there are alternative approaches for modelling nonlocal phenomena on graphs, e.g.~the one based on the transformed $k$-path Laplacian presented in~\cite{Estrada17, Estrada18}: this model shares some properties with fractional diffusion, such as superdiffusive behaviour on infinite one-dimensional graphs. 

In this paper we focus on the computational aspect of fractional dynamics, in particular on the efficient computation of the solution to the fractional diffusion equation on both undirected and directed graphs using rational Krylov methods. The solution at time $t \ge 0$ can be expressed as $\vec u(t) = f(L^T) \vec u(0)$, where $f(z) = \exp(-t z^\alpha)$ for $\alpha \in (0,1)$. The function $f$ has a branch cut on the negative real axis $(-\infty, 0]$, and hence it is not analytic in a neighbourhood of the spectrum of the graph Laplacian $L$, which is a singular $M$-matrix. The lack of a region of analyticity around the spectrum of $L$ causes the error bounds for Krylov methods based on the spectrum and the field of values to be unusable, and in practice this also negatively affects the actual convergence rate of these methods. Moreover, since $f$ is not analytic, it is preferable to approximate $f(L^T) \vec u(0)$ using rational Krylov methods instead of polynomial methods. In our experiments, in addition to the polynomial Krylov method, we investigate the performance of the Shift-and-Invert Krylov method with two different shifts, and of a rational Krylov method with asymptotically optimal poles presented in \cite{MasseiRobol20} for Laplace-Stieltjes functions.

To improve convergence, we propose to remove the singularity of the graph Laplacian with either a rank-one shift or a projection on a subspace. This enables us to use any rational Krylov method on the (now nonsingular) modified Laplacian, which exhibit faster convergence, and then recover the solution to the original problem with only little additional cost. We also show that the improved convergence of the projection method can be achieved without explicitly projecting the graph Laplacian, by suitably manipulating the initial vector (see \cref{sec:implicit-projection}). These techniques can be applied with no preliminary computations to any undirected graph, and they only require the solution of the singular linear system $L^T \vec z = \zero$ for a general digraph.

The paper is organized as follows. First, in \cref{sec:background-and-notation} we provide the necessary background and notation on graphs, and we introduce the graph Laplacian $L$. In \cref{sec:fractional-laplacian} we define the fractional powers $L^\alpha$, $\alpha \in (0,1)$, and we briefly discuss the properties of the fractional Laplacian. In \cref{sec:krylov-methods} we introduce rational Krylov methods for the computation of matrix functions, and in \cref{sec:desingularization} we present some techniques to remove the singularity of the graph Laplacian. In \cref{sec:numerical-experiments} we conduct some numerical experiments on real-world networks to compare the performance of different Krylov methods and demostrate the effectiveness of the desingulatization techniques proposed in~\cref{sec:desingularization}.

\section{Background and notation regarding graphs}
\label{sec:background-and-notation}
A \emph{directed graph}, or \emph{digraph}, is a pair $G = (V,E)$, where $V = \{ v_1, \dots, v_n \}$ is a set of $n$ nodes or vertices, and $E \subseteq V \times V$ is the set of edges. A digraph can be represented with its \emph{adjacency matrix} $A$, an $n \times n$ matrix whose entries are
\begin{equation*}
A_{ij} = \begin{cases}
1 & \text{if $(i,j) \in E$,} \\
0 & \text{otherwise.}
\end{cases}
\end{equation*}
The out-degree $d_i$ of a node $i$ is defined as the number of edges going out from $i$, i.e.~edges of the form $(i, \ell) \in E$ for some node $\ell \in V$. The vector $d$ of out-degrees can be computed as $\vec d = A \one$, where $\one$ denotes the all-ones vector of size $n$.
If we denote by $D = \diag(\vec d)$ the diagonal matrix of out-degrees, the (out-degree) \emph{graph Laplacian} of $G$ is defined as $L = D - A$. 

Note that one can also define the vector of in-degrees as $\vec d_\text{in} = A^T \one$, and the in-degree graph Laplacian as $L_\text{in} = \diag(\vec d_\text{in}) - A$. Here, we focus solely on the out-degree graph Laplacian $L$, and we refer to it as the graph Laplacian whenever there is no ambiguity. Most of the properties of $L$ also hold for $L_\text{in}$, and what we say for the out-degree graph Laplacian can be extended to the case of the in-degree Laplacian with only minor adjustments.

One can also consider weighted graphs, where to each edge $(i,j) \in E$ is associated a positive weight $w_{ij}$, repesenting the strength of the connection between nodes $i$ and $j$; if $(i,j) \notin E$, we write $w_{ij} = 0$. The matrix $W = (w_{ij})$ is a weighted adjacency matrix associated to $G$, and it can be used to defined a weighted vector of out-degrees $\vec d_W = W \one$ and a weighted graph Laplacian $L_W = \diag(\vec d_W) - W$. In this paper we only consider unweighted graphs for simplicity, but the techniques that we propose can be applied in the same way to weighted graphs. We also mainly focus on strongly connected graphs, i.e., graphs on which there exists a directed path from node $i$ to node $j$ for any pair of nodes $(i,j)$.
Recall that a graph is strongly connected if and only if its adjacency matrix $A$ (and hence also $L$) is irreducible, i.e., there exists no permutation matrix $P$ such that $P^T A P$ is block triangular.

\subsection{Properties of the graph Laplacian}
Here we briefly discuss some properties of the graph Laplacian $L$, which later will be used in the definition of its fractional powers. We also introduce the classical diffusion equation on graphs.

It follows from its definition that the graph Laplacian $L$ is a singular matrix, indeed it holds $L \one = \zero$. More specifically, the graph Laplacian is a singular $M$-matrix.

\begin{definition}[$M$-matrix \cite{BermanPlemmons87}]
A matrix $A \in \R^{n \times n}$ is an $M$-matrix if it holds $A = s I - B$ for some nonnegative matrix $B$, where $s \ge \rho(B)$, the spectral radius of $B$. It is a singular $M$-matrix if $s = \rho(B)$.
\end{definition}

One can easily prove the following basic result.

\begin{proposition}
\label{prop:laplacian-properties}
The graph Laplacian $L$ of a digraph $G$ has the following properties:
\begin{itemize}
\item
$L$ is a singular $M$-matrix.
\item
The nonzero eigenvalues of $L$ have positive real part.
\item
$0$ is a semisimple eigenvalue of $L$, i.e.~its algebraic multiplicity and geometric multiplicity are the same.
\end{itemize}
\end{proposition}

These properties are fundamental for being able to define fractional powers of $L$, as we will see shortly.

The graph Laplacian is used as the coefficient matrix in the diffusion equation on the graph $G$. Denote by $\vec u(t) \in \R^n$ a vector of concentrations at time $t$ of a substance that is diffusing on the graph. Up to normalization, we can assume that $\vec u(t)$ is a probability vector, i.e.~that $\vec u(t) \ge 0$ and $\vec u(t)^T \one = 1$. The diffusion equation on a directed graph reads
\begin{equation}
\label{eqn:diffusion-classical}
\begin{cases}
\dfrac{d}{dt}\vec u(t)^T = - \vec u(t)^T L,& \quad t \in (0, T],\\
\vec u(0) = \vec u_0 \ge 0,& \quad \vec u_0^T \one = 1,
\end{cases}
\end{equation}
and the solution to this system of ordinary differential equations can be explicitly stated in terms of the matrix exponential,
\begin{equation*}
\vec u(t)^T = \vec u_0^T e^{-t L}.
\end{equation*}
Using properties of $M$-matrices, one can easily prove that $e^{-tL}$ is a stochastic matrix, i.e.~that it has nonnegative entries and $e^{-tL} \one = \one$, and hence the solution $\vec u(t)$ is a probability vector at all times $t \ge 0$. Note that this property would not be preserved if we used column vectors instead of row vectors in \cref{eqn:diffusion-classical}: see, e.g., the discussion in~\cite{ChapmanMesbahi11}.

\section{The fractional graph Laplacian}
\label{sec:fractional-laplacian}
In this section, we recall the general definition of a matrix function in terms of the Jordan canonical form, following \cite[Section~1.2]{Higham08}, and we use it to define the fractional graph Laplacian and the related fractional diffusion process.

Recall that any matrix $A \in \C^{n \times n}$ can be expressed in the Jordan canonical form as
\begin{equation}
\label{eqn:Jordan}
\begin{aligned}
Z^{-1} A Z &= J = \diag(J_1, \dots, J_p), \\
J_k =\begin{bmatrix}
	 \lambda_k \\
\end{bmatrix}
\quad \text{or} \quad J_k &= \begin{bmatrix}
	 \lambda_k & 1 & & \\
	   & \lambda_k & \ddots & \\
	   & & \ddots & 1 \\
	   & & & \lambda_k
\end{bmatrix} \in \C^{m_k \times m_k},
\end{aligned}
\end{equation}
where $Z$ is nonsingular and $m_1 + m_2 + \dots + m_p = n$. An eigenvalue $\lambda$ is semisimple if and only if all the Jordan blocks associated to $\lambda$ are $1 \times 1$. We have the following definition.
\begin{definition}
The function $f$ is said to be defined on the spectrum of $A$ if the values
\begin{equation*}
f^{(j)}(\lambda_i), \qquad j = 0, \dots, m_i-1, \quad i = 1, \dots, n,
\end{equation*}
exist, where $f^{(0)} = f$ and $f^{(j)}$ is the $j$-th derivative of $f$.
\end{definition}
Provided that $f$ is defined on the spectrum of $A$, the matrix function $f(A)$ can be defined for any matrix using the Jordan canonical form.
\begin{definition}
\label{defn:matrix-function}
Let $f$ be defined on the spectrum of $A \in \C^{n \times n}$, and let $A$ have the Jordan canonical form \cref{eqn:Jordan}. Then we define
\begin{equation*}
f(A) = Z f(J) Z^{-1} = Z \diag \big(f(J_1), \dots, f(J_p)\big)Z^{-1},
\end{equation*}
where 
\begin{equation*}
f(J_k) = \begin{bmatrix}
	f(\lambda_k)
\end{bmatrix} \quad \text{or} \quad
f(J_k) = \begin{bmatrix}
	f(\lambda_k) & f'(\lambda_k) & \dots & \dfrac{f^{(m_k-1)}(\lambda_k)}{(m_k-1)!} \\
	& f(\lambda_k) & \ddots & \vdots \\
	& & \ddots & f'(\lambda_k) \\ 
	& & & f(\lambda_k)
\end{bmatrix}.
\end{equation*}
\end{definition}
By \cref{prop:laplacian-properties}, the function $f(z) = z^\alpha$, $\alpha \in (0,1)$ is defined on the spectrum of the graph Laplacian $L$, since the eigenvalue $0$ is semisimple and all other eigenvalues are in the right half-plane. Here we denote by $z^\alpha$ the branch of the fractional power with a cut on the negative real axis $(-\infty, 0]$, i.e.~if $z = \rho e^{i \theta}$, with $\rho > 0$ and $\theta \in (-\pi, \pi)$, then $z^\alpha = \rho^\alpha e^{i \alpha \theta}$.

With the above definition, the fractional Laplacian $L^\alpha$ is still an $M$-matrix. Indeed, we have the following result.
\begin{theorem}[\cite{BBDS20}]
If $A$ is a singular $M$-matrix with a semisimple $0$ eigenvalue, then $A^\alpha$ is a singular $M$-matrix for every $\alpha \in (0,1]$.
\end{theorem}

Moreover, since $L^\alpha \one = \zero$, we can interpret the fractional graph Laplacian as the Laplacian of a weighted graph on the same set of nodes as $G$, and we can use it to define a fractional diffusion process on $G$, with a system of differential equations analogous to \cref{eqn:diffusion-classical}: 
\begin{equation}
\label{eqn:diffusion-fractional}
\begin{cases}
\dfrac{d}{dt} \vec u(t)^T = -\vec u(t)^T L^\alpha ,& \quad t \in (0, T],\\
\vec u(0) = \vec u_0 \ge 0,& \quad \vec u_0^T \one = 1.
\end{cases}
\end{equation}
The solution of this system can be explicitly written in the form 
\begin{equation}
\label{eqn:diffusion-fractional-solution}
\vec u(t)^T = \vec u_0^T e^{-t L^\alpha}, \qquad t \ge 0.
\end{equation}
As in the case of classical diffusion, the solution $\vec u(t)$ to \cref{eqn:diffusion-fractional} is a probability vector at all times $t \ge 0$. 

\section{Rational Krylov methods}
\label{sec:krylov-methods}
In this section we briefly introduce rational Krylov methods for the computation of expressions of the form $f(A) \vec b$, with the goal of applying them for the computation the solution to the fractional diffusion equation \cref{eqn:diffusion-fractional}, which can be expressed in the form $\vec u(t) = f(L^T) \vec u_0$, where $f(z) = e^{-t z^\alpha}$. For a more extensive treatment of rational Krylov methods, including the problem of the selection of poles, we refer, e.g., to~\cite{GuettelThesis, Guettel13}. 

In many applications it is required to compute the product $f(A) \vec b$, where $A$ is a large and sparse matrix. In these cases, the computation of $f(A) \vec b$ by first computing the whole matrix $f(A)$ and then forming the product $f(A) \vec b$ is extremely expensive and often unfeasible; moreover, the matrix function $f(A)$ is generally dense even when the original matrix $A$ is sparse, making the full computation of $f(A)$ costly also in terms of storage for large scale problems. A rational Krylov methods overcomes these difficulties by directly approximating the product $f(A) \vec b$ using a low-dimensional search space, without explicitly computing $f(A)$. In each iteration of a rational Krylov method it is required to solve a shifted linear system involving $A$, making the iterations more expensive than those of a polynomial Krylov method, which only requires a matrix-vector product with $A$ at each iteration. However, the increased cost per iteration of a rational method is offset by the often superior approximation properties of rational functions, compared to polynomial approximation, especially for functions that are not analytic.

For any $k \ge 1$, define the \emph{rational Krylov subspace} of order $k$ associated to $A$ and $\vec b$ as
\begin{equation*}
\mathcal{Q}_k(A, \vec b) = q_{k-1}(A)^{-1} \vspan \{ \vec b, A \vec b, \dots, A^{k-1} \vec b \},
\end{equation*}
where $q_{k-1}(z) = \displaystyle\prod_{j = 1}^{k-1} (1 - z/\xi_j)$ is a polynomial identified by the \emph{poles} $(\xi_k)_{k \ge 1}$, which are located in $(\C \cup \infty)\setminus (\sigma(A) \cup \{0\})$. 
If all poles are equal to $\infty$, the rational Krylov subspace $\mathcal{Q}_k(A, \vec b)$ reduces to the \emph{polynomial Krylov subspace}
\begin{equation*}
\mathcal{P}_k(A, \vec b) = \vspan\{ \vec b, A \vec b, \dots, A^{k-1}\vec b\} = \{ p_{k-1}(A) \vec b \,:\, p \in \Pi_{k-1} \},
\end{equation*}
where $\Pi_{k-1}$ denotes the set of polynomials of degree $\le k-1$.

The rational Krylov subspaces $\mathcal{Q}_k(A,\vec b)$ form a sequence of nested subspaces, each of dimension $k$, as long as $k \le K$, where $K$ is the invariance index of the sequence, i.e.~the smallest index such that $\mathcal{Q}_K(A, \vec b) = \mathcal{Q}_{K+1}(A,\vec b)=\mathcal{Q}_{K+j}(A,\vec b)$ for all $j \ge 0$.

Generally it is assumed that $k \le K$. If this is the case, we can compute an orthonormal basis of $\mathcal{Q}_k(A, \vec b)$ using Ruhe's rational Arnoldi algorithm \cite{Ruhe94}, which is summarized in \cref{alg:Arnoldi}. The first basis vector is chosen as $\vec v_1 = \vec b/\norm{ \vec b}_2$. After $j$ iterations, the columns of the matrix $V_j = [\vec v_1, \dots, \vec v_j] \in \C^{n\times j}$ form an orthonormal basis of $\mathcal{Q}_j(A, \vec b)$. To construct a basis of $\mathcal{Q}_{j+1}(A, \vec b)$, the Arnoldi algorithm makes use of a \emph{continuation vector} $\vec w_j \in \mathcal{Q}_{j}(A, \vec b)$ such that $(A-\xi_{j+1}I)^{-1} A \vec w_j \in \mathcal{Q}_{j+1}(A, \vec b) \setminus \mathcal{Q}_j(A, \vec b)$. As is customary, we assume that the last computed basis vector $ \vec v_j$ is a continuation vector, and we compute $ \vec v_{j+1}$ by orthonormalizing $\vec w_j = (I - A/\xi_j)^{-1}A \vec v_j$ against all the previous basis vectors.

\begin{algorithm}
\hspace*{\algorithmicindent} \textbf{Input}: $A$, $\vec b$, $k$, $\{\xi_1, \dots, \xi_k\}$ \\
\hspace*{\algorithmicindent} \textbf{Output}: $\{\vec v_1, \dots, \vec v_k\}$ 
\begin{algorithmic}[1]
%~ \Function{RatArnoldi}{$A,b,k, \{\xi_1, \dots, \xi_k\}$}
\State $\vec v_1 \gets \vec b / \norm{\vec b}_2$
\For{$j = 1, \dots, k-1$}
	\State $\vec w_j \gets (I - A/\xi_j)^{-1}A \vec v_j$
    \For{$i = 1, \dots, j$}
		\State $h_{ij} \gets \vec w_j^T \vec v_i$
		\State $\vec w_j \gets \vec w_j - h_{ij} \vec v_j$
	\EndFor
	\State $h_{j+1, j} \gets \norm{\vec w_j}_2$
	\If{$h_{j+1, j} = 0$} 
		 stop
	\Else
		\State $\vec v_{j+1} \gets \vec w_j/h_{j+1, j}$
	\EndIf
\EndFor
%~ \State \textbf{return} $v_1, \dots, v_k$
%~ \EndFunction
\end{algorithmic}
\caption{Rational Arnoldi algorithm for computing an orthonormal basis of $\rat_k(A,b)$.}
\label{alg:Arnoldi}
\end{algorithm}

An approximation to $\vec y = f(A) \vec b$ in the subspace $\mathcal{Q}_k(A,\vec b)$ can be computed as
\begin{equation}
\label{eqn:rational-Krylov-approx}
\barvec y_k = V_k f(V_k^* A V_k) V_k^* \vec b.
\end{equation}
For the solution of the fractional diffusion problem, we are going to use real poles (in particular, located on $(-\infty, 0)$), so the matrices $V_k$ are going to be real.

When the sequence of poles $(\xi_k)_{k \ge 1}$ consists of a single repeated pole $\xi$, the rational Krylov subspaces that are generated are known as \emph{Shift-and-Invert Krylov subspaces}, and they can be written more simply as
\begin{equation*}
\mathcal{S}_k(A,\vec b) := (A - \xi I)^{-k+1} \vspan\{ \vec b, \dots, A^{k-1} \vec b \} = \mathcal{K}_k((A - \xi I)^{-1}, \vec b),
\end{equation*}
i.e.~it corresponds to a polynomial Krylov subspace relative to the matrix $(A - \xi I)^{-1}$.

The Shift-and-Invert method was first investigated for the approximation of matrix functions in~\cite{MoretNovati04, VDEHochbruck06}. 
Even though this type of Krylov subspace sacrifices some flexibility in the choice of the poles, it is appealing because at each iteration we are required to solve a linear system with the same matrix $(A - \xi I)$; this allows us, for instance, to compute an LU factorization of $(A - \xi I)$ only once, and then we can apply it at each iteration to solve the linear systems at a reduced cost. Therefore, although a Shift-and-Invert method will typically require more iterations to converge compared to a rational Krylov method with a carefully chosen sequence of poles, it can still be competitive in terms of execution time.

%~ Another specific rational Krylov method that is often used is the \emph{extended Krylov subspace method}~\cite{DruskinKnizhnerman98}, in which the poles alternate between $0$ and $\infty$. However, this method is not suited for in our setting, since the graph Laplacian is a singular matrix, and hence we do not discuss it here.

The effectiveness of the approximation to $f(A) \vec b$ given by \cref{eqn:rational-Krylov-approx} can be related to the problem of approximating $f$ with rational functions. Denote by $\W(A)$ the \emph{field of values} of $A$ \cite[Definition~1.1.1]{HornJohnson91}, i.e.~the set
\begin{equation*}
\W(A) = \{ \vec x^*A \vec x : \vec x \in \C^n,\, \norm{\vec x}_2 = 1 \}. 
\end{equation*}
The field of values, also known as numerical range, is a convex and compact set which contains the spectrum $\sigma(A)$, and it reduces to the convex hull of $\sigma(A)$ when $A$ is a normal matrix. 

The following theorem by Crouzeix and Palencia provides a bound for the norm of a matrix function using the field of values $\W(A)$.
\begin{theorem}[\cite{Crouzeix07, CrouzeixPalencia17}]
\label{thm:Crouzeix}
Assume that $f$ is analytic in a neighbourhood of $\W(A)$. Then it holds
\begin{equation*}
\norm{f(A)}_2 \le C \norm{f}_{\W(A)}, \qquad C \le 1 + \sqrt{2},
\end{equation*}
where $\norm{f}_E = \sup_{z \in E} \abs{f(z)}$ for any set $E \subset \C$.
\end{theorem} 
A conjecture by Crouzeix states that the inequality $\norm{f(A)}_2 \le C \norm{f}_{\W(A)}$ holds with $C = 2$ for any matrix $A$.

With \cref{thm:Crouzeix} one can prove the following.

\begin{proposition}[{\cite[Corollary~3.4]{Guettel13}}]
\label{prop:rational-Krylov-bound}
Let $f$ be analytic in a neighbourhood of $\W(A)$. Then the rational Krylov approximation $\barvec y_k$ defined in~\cref{eqn:rational-Krylov-approx} satisfies
\begin{equation}
\label{eqn:rational-approximation-bound}
\norm{f(A) \vec b - \barvec y_k}_2 \le 2C \norm{\vec b}_2 \min_{p_{k-1} \in \Pi_{k-1}} \norm{f - p_{k-1}/q_{k-1}}_{\W(A)},
\end{equation}
where $\Pi_{k-1}$ denotes the set of polynomials of degree $\le k-1$.
\end{proposition}

The bound given by \cref{prop:rational-Krylov-bound} decays rapidly to zero when $f$ is an entire function (e.g., $f(z) = e^z$), or if it has a large region of analiticity surrounding the field of values of $A$. Unfortunately, in the case of fractional diffusion on graphs, the $0$ eigenvalue of the Laplacian is located on the boundary of the region of analiticity of $f$. Moreover, for most directed graphs the field of values of the Laplacian intersects the negative real axis $(-\infty, 0)$, preventing us from using convergence results based on the field of values. The presence of an eigenvalue at $0$ can also be detrimental in practice for the convergence of Krylov methods. With this motivation in mind, in \cref{sec:desingularization} we propose some techniques to remove the singularity of the graph Laplacian, in order to work with nonsingular matrices and improve the convergence of Krylov methods.

\subsection{Laplace-Stieltjes functions}	
\label{sec:laplace-stieltjes-functions}
The problem of the selection of poles for rational Krylov methods is a highly active area of reseach, and many different choices have been proposed in the literature, depending on the function $f$ and on the spectrum of $A$ (see, for instance,~\cite{Guettel13}). Most of the existing analysis deals with real symmetric matrices, since in that case the field of values is reduced to an interval on the real line, and hence the minimization problem~\cref{eqn:rational-approximation-bound} becomes easier to handle. 

A pole selection strategy for the evaluation of $f(A) \vec b$ was recently proposed in~\cite{MasseiRobol20} for the case of a Hermitian positive definite matrix $A$ and a Cauchy-Stieltjes or Laplace-Stieltjes function $f$. 
%~ As we shall see shortly, the function $f(z) = e^{-tz^\alpha}$ that appears in \cref{eqn:diffusion-fractional-solution} is a Laplace-Stieltjes function, and hence we expect this choice of poles to perform well, at least in the case of undirected graphs. 
For a matrix $A$ with spectrum contained in the positive interval $[a,b]$, the choice of poles described in~\cite{MasseiRobol20} gives after $k$ iterations an error
\begin{equation}
\label{eqn:laplace-stieltjes-convergence-rate}
\norm{f(A)\vec b - \barvec y_k}_2 \sim O\big(\rho_{[a,b]}^{\frac{k}{2}}\big), \qquad \text{where} \quad \rho_{[a,b]} = \exp \big(-\pi^2/\log(\tfrac{4b}{a}) \big).
\end{equation}
However, the poles that satisfy the error bound for iteration $k+1$ are not obtained by adding a new pole to the ones of iteration $k$, so in order to effectively use this pole selection strategy one would need to decide a priori the number of iterations to be performed. In order to overcome this drawback, in~\cite[Section~3.5]{MasseiRobol20} the authors use the method of equidistributed sequences (EDS) to construct an infinite sequence of poles with the same asymptotic rate of convergence, that can be more easily used in practice. For the details on the construction of this pole sequence, we refer to the discussion in~\cite[Section~3.5]{MasseiRobol20}.

In this section we observe that the function $f(z) = e^{-t z^\alpha}$ is a Laplace-Stieltjes function, and hence we can use the pole sequence proposed in \cite{MasseiRobol20} for the fractional diffusion problem \cref{eqn:diffusion-fractional}. Even though there are no guarantees on the effectiveness of this pole sequence for general matrices, in our numerical experiments we observed that it provides a good convergence rate even when $A$ is the (singular and nonsymmetric) Laplacian of a directed graph.

\begin{definition}
\label{defn:laplace-stieltjes-function}
A function $f: (0, \infty) \to \R$ is a Laplace-Stieltjes function if it can be expressed in the form
\begin{equation*}
f(z) = \int_0^\infty e^{-t z} d\mu(t),
\end{equation*}
where $\mu$ is a positive measure on $[0, \infty)$.
\end{definition}

The class of Laplace-Stieltjes functions coincides with the class of \emph{completely monotonic functions}, i.e.~infinitely differentiable functions defined on $(0,\infty)$ such that 
\begin{equation}
\label{eqn:completely-monotone-condition}
(-1)^k f^{(k)}(z) \ge 0 \qquad \forall\, z > 0 \quad \text{and} \quad k \ge 0.
\end{equation}
The equivalence between these two classes of functions is known as Bernstein's theorem~\cite[Theorem~1.4]{SSV-BernsteinFunctions}. 

A class of functions which is closely related to completely monotonic functions is the class of \emph{Bernstein functions}, that consists of all functions $f : (0, \infty) \to \R$ of class $C^\infty$ such that 
\begin{equation}
\label{eqn:Bernstein-condition}
f(z) \ge 0 \quad \text{and} \quad (-1)^{k-1}f^{(k)} (z) \ge 0, \qquad \forall\, k \ge 1 \quad \text{and} \quad \forall \, z > 0.
\end{equation}
Observe that a nonnegative function $f : (0,\infty) \to \R$ is a Bernstein function if and only if $f'$ is a completely monotonic function. The fractional power $f(z) = z^\alpha$, for $\alpha \in (0,1)$, is an example of a Bernstein function. By \cite[Theorem~3.7]{SSV-BernsteinFunctions}, if $f$ is a positive Bernstein function, then the function $g(z) = e^{-t f(z)}$ is completely monotonic for all $t > 0$. This proves that $g(z) = e^{-t z^\alpha}$ is a completely monotonic (equivalently, Laplace-Stieltjes) function for all $t > 0$ and $\alpha \in (0,1)$. This fact can also be easily proved diredtly by computing the derivatives of $g$ and checking that condition \cref{eqn:completely-monotone-condition} is verified.

\section{Dealing with the singularity}
\label{sec:desingularization}

As we have discussed previously, the functions $f(z) = z^\alpha$ and $g(z) = e^{-t z^\alpha}$ that are involved in fractional dynamics are not analytic at $z=0$. Since the graph Laplacian $L$ always has a zero eigenvalue, the convergence of rational Krylov methods for the computation of $f(L^T) \vec b$ and $g(L^T) \vec b$ may be hindered by the fact that the function has no region of analyticity surrounding the spectrum of $L$.

In this section we propose a rank-one shift and a subspace projection that can be used to transform the graph Laplacian into a nonsingular matrix, and we provide simple formulas that link $f(L)$ and $f(L^T)\vec b$ with functions of the transformed matrix. We are also going to show that Krylov methods directly applied to the singular graph Laplacian can inherit the improved convergence of the projection approach, at least in exact arithmetic, provided that the initial vector $\vec b$ is suitably modified. 

We present these techniques in detail for strongly connected directed graphs. Recall that in this case the eigenvalue $0$ of the Laplacian $L$ is simple.

\subsection{Rank-one shift}

%~ Let $L \in \R^{n \times n}$, and assume that we know the rignt and left eigenvectors associated to a simple eigenvalue $\lambda \in \sigma(L)$, i.e. we have $x$, $y \in \R^n$ such that $y^T L = \lambda y^T$ and $L x = \lambda x$. Assume that $x^T y \ne 0$, and normalize $x$ and $y$ so that $y^T x = 1$. 

Recall that $L \one = \zero$, and let $\vec z > \zero$ be such that $\vec z^T L = \zero^T$ and $\vec z^T \one = 1$ (the positivity of $\vec z$ is a consequence of the Perron-Frobenius Theorem \cite{Meyer00}). The vector $\vec z$ is uniquely defined by the above identities if the graph $G$ is strongly connected.

The right and left eigenvectors $\one$ and $\vec z$ can be respectively completed to a right and left Jordan basis for $L$ with two matrices $R,\, S \in \C^{n \times (n-1)}$, so that we have
\begin{equation*}
Z^{-1} L Z = J = \begin{bmatrix}
	 0 & 0 \\
	 0 & J_1
\end{bmatrix}, \qquad \text{where } Z = \begin{bmatrix}
	  \one & R \\
\end{bmatrix} \quad \text{and} \quad Z^{-1} = \begin{bmatrix}
	 \vec z^T \\
	 S^T
\end{bmatrix}.
\end{equation*}
The matrix $J_1 \in \C^{(n-1) \times (n-1)}$ contains all the other Jordan blocks of $L$, which correspond to nonzero eigenvalues. 
 
Now, denoting by $\vec e_1 \in \R^n$ the first vector of the canonical basis, observe that $\one \vec z^T = Z \vec e_1 \vec e_1^T Z^{-1}$, and hence for all $\theta \in \R$ we have 
\begin{equation*}
L + \theta \one \vec z^T = Z \begin{bmatrix}
	 \theta & 0 \\
	 0 & J_1
\end{bmatrix}Z^{-1},
\end{equation*}
i.e.~the matrix $L + \theta \one \vec z^T$ has the same spectrum as $L$ except for the eigenvalue $0$, which is replaced by $\theta$. Therefore, using basic properties of matrix functions, for any function $f$ defined on the spectra of $L$ and $L + \theta \one \vec z^T$ it holds
\begin{equation}
\label{eqn:rank-one-shift-identity}
f(L + \theta \one \vec z^T) = Z \begin{bmatrix}
	 f(\theta) & 0 \\
	 0 & f(J_1) 
\end{bmatrix} Z^{-1} = f(L) + [f(\theta) - f(0)] \one \vec z^T.
\end{equation}

Identity \cref{eqn:rank-one-shift-identity} allows us to compute $f(L + \theta \one \vec z^T)$ instead of $f(L)$ and then recover the latter for a minimal cost. For any $\theta > 0$ (e.g.,~$\theta = 1$), the matrix $L + \theta \one \vec z^T$ is nonsingular and all its eigenvalues have strictly positive real part, so we expect Krylov methods to converge faster when $f$ has a branch cut on $(-\infty, 0]$, e.g.~for $f(z) = z^{\alpha}$.

In particular, for fractional diffusion the objective is the computation of $f(L^T)\vec b$ for a probability vector $\vec b$ and $f(z) = e^{-t z^\alpha}$, and identity \cref{eqn:rank-one-shift-identity} becomes 
\begin{equation}
\label{eqn:rank-one-shift-krylov}
f(L^T)\vec b = f(L^T + \theta \vec z \one^T)\vec b + [f(0) - f(\theta)]\vec z.
\end{equation}

When the matrix $L$ is large and sparse and a rational Krylov method is used to approximate $f(L^T + \theta \vec z \one^T)$, it would be preferable to solve shifted linear systems involving the dense matrix $L^T + \theta \vec z \one^T$ without explicitly forming it. This can be done by using the Sherman-Morrison formula: for an invertible matrix $A$ and two vectors $\vec u$, $\vec v$ such that $1 + \vec v^T A^{-1} \vec u \ne 0$, the matrix $A + \vec u \vec v^T$ is invertible and it holds
\begin{equation}
\label{eqn:Sherman-Morrison-formula}
(A + \vec u \vec v^T)^{-1} = A^{-1} - \frac{A^{-1} \vec u \vec v^T A^{-1}}{1 + \vec v^T A^{-1} \vec u}.
\end{equation}
In our setting, for a pole $\xi \in (-\infty, 0)$ the invertibility condition $1 + \one^T (L^T - \xi I)^{-1} \vec z \ne 0$ is always satisfied (since ($L^T - \xi I)^{-1} \ge 0$, being the inverse of a nonsingular $M$-matrix), and identity \cref{eqn:Sherman-Morrison-formula} becomes
\begin{equation}
\label{eqn:Sherman-Morrison-laplacian}
(L^T + \theta \vec z \one^T -\xi I)^{-1} = (L^T - \xi I)^{-1} + \frac{\theta}{\xi(\theta-\xi)} \vec z \one^T.
\end{equation}

\begin{remark}
\label{rem:previous-use-of-sherman-morrison}
We mention that the Sherman-Morrison formula has already been applied in the literature in the context of rational Krylov methods. For instance, in \cite[Section~3.1]{ShankSimoncini13} the authors use the Sherman-Morrison-Woodbury formula in the construction of an ``augmented'' Krylov subspace associated to a singular matrix, arising in connection with the solution of a constrained Sylvester equation. 
\end{remark}

\begin{remark}
\label{rem:Sherman-Morrison-small-xi}
Using the Jordan canonical form, it is not difficult to see that it holds 
\begin{equation}
\label{eqn:small-xi-limit}
\lim_{\xi \to 0^{-}} \xi (L^T -\xi I)^{-1} = - \vec z \one^T,
\end{equation}
and hence for small $\xi < 0$ and any vector $\vec w$ the identity \cref{eqn:Sherman-Morrison-laplacian} becomes 
\begin{align}
\label{eqn:Sherman-Morrison-cancellation}
(L^T + \theta \vec z \one^T -\xi I)^{-1} \vec w &= (L^T - \xi I)^{-1}\vec w + \frac{\theta}{\xi(\theta-\xi)} (\one^T \vec w) \vec z \\
\nonumber
&\approx -\xi^{-1} (\one^T \vec w) \vec z + \frac{\theta}{\xi(\theta-\xi)} (\one^T \vec w) \vec z, 
\end{align}
i.e. we are very close to subtracting two multiples of the vector $\vec z$ of approximately the same length. Hence formula \cref{eqn:Sherman-Morrison-cancellation} may suffer from severe numerical cancellation when the pole $\xi$ is very close to the origin, and therefore its use is not advised in that situation. Indeed, in our numerical experiments we observed a large loss of precision when solving linear systems with formula \cref{eqn:Sherman-Morrison-cancellation} for poles $\xi < 0$ of order $10^{-6}$.
\end{remark}

In order to address the issue mentioned in \cref{rem:Sherman-Morrison-small-xi}, we now derive an alternative way to compute the solution of the shifted linear system $(L^T - \xi I ) \vec \phi = \vec w$, in order to avoid the cancellation in \cref{eqn:Sherman-Morrison-cancellation} when $\xi$ is small. We have 
\begin{equation*}
\one^T \vec w = \one^T (L^T - \xi I) \vec \phi = -\xi \, \one^T \vec \phi \:\implies\: \one^T \vec \phi = -\xi^{-1} \one^T \vec w,
\end{equation*}
so we can define $\vec \psi := \vec \phi + \xi^{-1} (\one^T \vec w) \vec z$, and it holds by construction that $\one^T \vec \psi = 0$. It is also straightforward to verify that $\vec \psi$ is the solution to the linear system
\begin{equation*}
(L^T - \xi I) \vec \psi = \vec w - (\one^T \vec w) \vec z,
\end{equation*}
and that the vector $\vec w - (\one^T \vec w) \vec z$ is orthogonal to $\one$. Hence, we can compute $\vec \phi$ as
\begin{equation}
\label{eqn:cancellation--phi}
\vec \phi = \vec \psi - \xi^{-1} (\one^T \vec w) \vec z, \qquad \text{where} \quad (L^T - \xi I) \vec \psi = \vec w - (\one^T \vec w)\vec z.
\end{equation}
With formula \cref{eqn:cancellation--phi}, we have explicitly separated a component of the solution that is proportional to $\xi^{-1} \vec z$. By substituting \cref{eqn:cancellation--phi} in \cref{eqn:Sherman-Morrison-cancellation}, we obtain
\begin{equation}
\label{eqn:cancellation--final}
\begin{aligned}
(L^T + \theta \vec z \one^T - \xi I)^{-1} \vec w &= \vec \phi + \frac{\theta}{\xi(\theta - \xi)} (\one^T \vec w) \vec z \\
&= \vec \psi - \xi^{-1} (\one^T \vec w) \vec z + \frac{\theta}{\xi(\theta - \xi)} (\one^T \vec w) \vec z \\
&= \vec \psi + \frac{\one^T \vec w}{\theta - \xi} \, \vec z.
\end{aligned}
\end{equation}
Observe that cancellation is avoided when using \cref{eqn:cancellation--final}, because the subtraction of the two close multiples of the vector $\vec z$ is performed analytically. Moreover, because of \cref{eqn:small-xi-limit} and since $(\vec w - (\one^T \vec w)\vec z) \perp \one$,  we have
\begin{equation*}
\lim_{\xi \to 0^-} \xi (L^T - \xi I)^{-1} (\vec w - (\one^T \vec w)\vec z) = \zero,
\end{equation*}
so for small $\xi < 0$ we do not expect $\vec \psi$ to have a component of order $\xi^{-1}$ along the vector $\vec z$ (note that, in general, this argument fails for $\vec \phi$). Our numerical experiments confirm that the use of formula \cref{eqn:cancellation--final} fixes the problem of cancellation.

\begin{remark}
\label{rem:solution-of-singular-linear-system}
Note that if the graph is undirected, or more generally if it is balanced (i.e.~each node has equal in- and out-degree), we also have $\vec z = \one$ up to a normalization factor, so  no preliminary computation is needed to use this approach with the rank-one shift. On the other hand, for a general digraph it is first required to compute a nonzero vector $\vec z$ such that $L^T \vec z = \zero$. 

The problem of solving this linear system was recently discussed, for instance, in~\cite{BenziFikaMitrouli19}. 
One possible approach is to compute an LDU factorization of the transpose of the graph Laplacian, $L^T = \mathcal{L}\mathcal{D}\mathcal{U}$, where $\mathcal{L}$ is unit lower triangular, $\mathcal{U}$ is unit upper triangular, and $\mathcal{D}$ is diagonal with $\mathcal{D}_{ii} > 0$ for $i = 1, \dots, n-1$ and $\mathcal{D}_{nn} = 0$. Such a factorization always exists since $L$ is an irreducible singular $M$-matrix \cite{BermanPlemmons87}, and it can be computed in a stable way using Gaussian elimination, with no pivoting required~\cite{FunderlicPlemmons81}. The original linear system $L^T \vec z = \zero$ is thus equivalent to $\mathcal{D} \mathcal{U} \vec z = \zero$, which can be solved by fixing $z_n = 1$ and solving the lower triangular linear system $\mathcal{U} \vec z = \vec e_n$ via backward substitution. We remark that $\mathcal{L}^{-1} \ge 0$, so the vector $\vec z$ is nonnegative, and it can be indeed normalized so that $\vec z^T \one = 1$. We also mention that when $L$ is sparse this method can be improved by computing the LDU factorization of $P^T L^T P$ instead of $L^T$, where $P$ is a permutation matrix suitably chosen to reduce the fill-in in the factors $\mathcal{L}$ and $\mathcal{U}$. For example, the Matlab routines \texttt{amd} and \texttt{symrcm} can be used for this purpose.

Alternatively, when $L$ is very large and sparse, the linear system $L^T \vec z = \zero$ can be solved iteratively, e.g.~with a preconditioned GMRES method (see \cite{BenziFikaMitrouli19} and references therein).
Of course, if $L$ is large and we choose to solve $L^T \vec z = \zero$ iteratively, we should also use an iterative method for the solution of the shifted linear systems at each step of the rational Krylov iteration. However, in this paper we do not address this specific subproblem, and we instead focus on the case where it is feasible to solve the linear systems with a direct method.
\end{remark}

\subsection{Projected Krylov methods}

Another way to obtain a nonsingular matrix from the graph Laplacian is to project $L$ on the $n-1$ dimensional subspace $\mathcal{S} = \vspan\{ \one \}^{\perp}$.
We remark that the approach we present here is similar to the one described in \cite[Section 4]{BurrageHaleKay12}, where the authors separate the eigenvalue $0$ from the rest of the spectrum of a symmetric positive semidefinite matrix $A$, to compute $f(A) \vec b$ with an integral on a contour surrounding $\sigma(A) \setminus \{ 0 \}$. See also \cite{IlicTurner04, IlicTurner08} for a discussion of more general \emph{spectral splitting} methods for symmetric matrices.

Let $\{ \vec q_1, \dots, \vec q_{n-1}\}$ be an orthonormal basis of $\mathcal{S}$, and define the $n \times (n-1)$ matrix  $Q := \begin{bmatrix}\vec q_1 &  \!\dots \! & \vec q_{n-1} \end{bmatrix}$. 
The matrix $\tilde Q := \begin{bmatrix}
	 Q & \frac{1}{\sqrt{n}}\one \\
\end{bmatrix}$ is orthogonal, and we have $Q^T Q = I_{n-1}$ and $QQ^T = I_n - \frac{1}{n}\one \one^T$. Here we denoted by $I_k$ the identity matrix of size $k \times k$ in order to stress that the two matrices have a different size; in the sequel, we will drop the subscript when there is no ambiguity. Observe that the matrix $Q^T L Q$ is nonsingular, since $\range Q = \vspan\{\one\}^\perp$, $\ker Q^T = \ker L = \vspan\{\one\}$ and $\range L = \vspan\{ \vec z\}^\perp$. 

We are going to rewrite $f(L)$ in terms of $f(Q^T L Q)$ by using some properties of matrix functions. Recalling that $L \one = \zero$ and that $\one^T Q = \zero^T$, we have:
\begin{equation*}
\tilde Q^T L \tilde Q = \begin{bmatrix}
	 Q^T \\
	 \frac{1}{\sqrt{n}}\one^T
\end{bmatrix}
L
\begin{bmatrix}
	 Q & \frac{1}{\sqrt{n}}\one
\end{bmatrix} = \begin{bmatrix}
	 Q^T L Q & \zero \\
	 \frac{1}{\sqrt{n}}\one^T L Q & 0
\end{bmatrix}.
\end{equation*}
Now, using well known properties of matrix functions, we have 
\begin{equation}
\label{eqn:projected-krylov-partial}
\tilde Q^T f(L) \tilde Q = f(\tilde Q^T L \tilde Q) = 
\begin{bmatrix}
	 f(Q^T L Q) & \zero \\
	 \vec \varphi^T & f(0)
\end{bmatrix},
\end{equation}
for some $\vec \varphi \in \R^{n-1}$. The vector $\vec \varphi$ can be expressed in closed form (see, e.g., \cite[Theorem~1.21]{Higham08}), but this is not necessary for our purposes.

Let us assume at first that our goal is to compute $f(L^T)\vec v$ for a vector $\vec v$ such that $\one^T \vec v = 0$.
Using \cref{eqn:projected-krylov-partial}, we get
\begin{align}
\notag
f(L^T)\vec v &= \tilde Q f(\tilde Q^T L \tilde Q)^T \tilde Q^T \vec v \\
\notag
&= \begin{bmatrix}
	 Q & \frac{1}{\sqrt{n}}\one
	 \end{bmatrix} \begin{bmatrix}
	 f(Q^T L^T Q) & \vec \varphi \\
	 \vec 0^T & f(0)
\end{bmatrix}
\begin{bmatrix}
	 Q^T \vec v \\
	 0
\end{bmatrix}\\
&= Q f(Q^T L^T Q) Q^T \vec v.
\label{eqn:projected-krylov-orthogonal}
\end{align}
Now, consider the computation of $f(L^T)\vec b$ for a generic vector $\vec b$. If $\one^T \vec b = \beta \ne 0$, we can always write $\vec b = \vec v + \beta \vec z$ for some vector $\vec v \perp \one$ (recall that $\vec z$ satisfies $L^T \vec z = \zero$ and $\one^T \vec z = 1$). Hence, using \cref{eqn:projected-krylov-orthogonal} we have
\begin{align}
\label{eqn:projected-krylov-general}
\notag
f(L^T) \vec b &= f(L^T) \vec v + \beta f(L^T) \vec z \\
&= Q f(Q^T L^T Q) Q^T \vec v + \beta f(0) \vec z.
\end{align}

Using \cref{eqn:projected-krylov-general}, we can compute $f(L^T)\vec b$ by using a rational Krylov method on the nonsingular projected matrix $Q^T L^T Q$. As the rank-one shift, this requires knowledge of $\vec z$, the left $0$-eigenvector of the graph Laplacian, which must be computed beforehand by solving the singular linear system $L^T \vec z = \zero$. In order to make this approach viable, we need to be able to compute matrix-vector products with $Q$ efficiently: we address this problem in \cref{sec:mat-vecs-with-Q}. We are also going to show that in exact arithmetic the Krylov methods for $f(L^T)\vec v$ and $Q f(Q^T L^T Q) Q^T \vec v$ construct precisely the same approximations after an equal number of iterations, so it is actually not necessary to perform the projection explicitly.

\subsubsection{Fast matrix-vector products with $Q$}
\label{sec:mat-vecs-with-Q}

In this part we show how to construct a matrix $Q$ with orthonormal columns spanning the subspace $\mathcal{S} = \vspan{\one}^\perp$, such that the matrix-vector products of the form $Q \vec u$ and $Q^T \vec v$ can be perfomed with cost $O(n)$.

Let us define the orthogonal matrix $\tilde Q = \begin{bmatrix}
	 Q & \frac{1}{\sqrt{n}}\one \\
\end{bmatrix}$ as
\begin{equation*}
\tilde Q = \begin{bmatrix}[cccc|c]
	1/\sqrt{n} & \cdots & \cdots & 1/\sqrt{n} & 1/\sqrt{n} \\
	\cmidrule(lr){1-5}
	    r      &   s    & \dots & s & 1/\sqrt{n} \\
	    s      &   r    & \ddots & \vdots & \vdots  \\
	    \vdots & \ddots & \ddots & s & \vdots \\
	    s      & \dots  & s & r & 1/\sqrt{n} \\
\end{bmatrix}, \qquad \begin{aligned} \\
\text{where} \qquad s &= \frac{1}{1-n}\Big(1 + \frac{1}{\sqrt{n}}\Big), \\
r &= s+1
\end{aligned}
\end{equation*} 
so that we have
\begin{equation}
\label{eqn:projection-Q}
Q = \frac{1}{\sqrt n} \begin{bmatrix}
	 1 \\
	 \zero_{n-1}
\end{bmatrix} \one_{n-1}^T + s \begin{bmatrix}
	 0 \\
	 \one_{n-1}
\end{bmatrix} \one_{n-1}^T + \begin{bmatrix}
	 \zero_{n-1}^T \\
	  I_{n-1}
\end{bmatrix},
\end{equation} 
where we denoted by $\zero_{n-1}$, $\one_{n-1}$ and $I_{n-1}$ respectively the all-zeroes vector, the all-ones vector, and the identity matrix of size $n-1$. It is straightforward to see that with the above definition $\tilde Q$ is indeed an orthogonal matrix.
Now, for any vector $\vec v \in \R^n$ and $\vec u \in \R^{n-1}$ we have
\begin{align}
\nonumber
&Q \vec u = \frac{1}{\sqrt n} \one_{n-1}^T \vec u \begin{bmatrix}
	 1 \\
	 0_{n-1}
\end{bmatrix}  + s \one_{n-1}^T \vec u \begin{bmatrix}
	 0 \\
	 \one_{n-1}
\end{bmatrix}  + \begin{bmatrix}
	 0 \\
	 \vec u
\end{bmatrix} \\
\label{eqn:projection-operations}
&Q^T \vec v = \frac{1}{\sqrt n} [\vec v]_1 \one_{n-1} + s \sum_{j = 2}^n [\vec v]_j \one_{n-1} + \begin{bmatrix}
	 v_2 \\
	 \vdots \\
	 v_n
\end{bmatrix}\\
\nonumber
&(Q^T L^T Q - \xi I_{n-1})^{-1} \vec u = Q^T (L^T - \xi I)^{-1} Q \vec u.
\end{align}
The last equality in \cref{eqn:projection-operations} follows from \cref{lemma:projected-kylov-properties}(b).

Hence the matrix-vector products $Q \vec u$ and $Q^T \vec v$ can be computed with cost $O(n)$, and the solution of a shifted linear system with $Q^T L^T Q$ can be reduced with cost $O(n)$ to the solution of a  shifted linear system with $L^T$. 
%~ For our numerical experiments with the projected variant of rational Krylov methods, we defined the matrix $Q$ as in~\cref{eqn:projection-Q} and we used the identities~\cref{eqn:projection-operations} to perform the operations involving $Q$ efficiently.

\subsubsection{Implicit projection}
\label{sec:implicit-projection}

In the following part, we are going to examine how rational Krylov methods for the computation of $f(L^T) \vec b$ are related to their projected counterpart, i.e.~to methods that first approximate $f(Q^T L^TQ) Q^T \vec b$ using rational Krylov subspaces and then use \eqref{eqn:projected-krylov-orthogonal} to compute $f(L^T) \vec b$, in the case of an initial vector $\vec b \perp \one$. Note that the assumption $\vec b \perp \one$ is not satisfied when computing the solution to the fractional diffusion equation, since in that case the initial vector $\vec u_0$ is a probability vector, and hence $\one^T \vec u_0 = 1$. However, with the same procedure used in identity~\cref{eqn:projected-krylov-general}, the results of this section can be used with minor modifications for any initial vector~$\vec b$.

We are going to prove our result in a slightly more general scenario: let $A \in \R^{n \times n}$, and let $\vec x$ be a left eigenvector of $A$ such that $\vec x^T A = \lambda \vec x^T$. The specific case of the graph Laplacian will then correspond to $A = L^T$, $\vec x = \one$ and $\lambda = 0$. Let $Q$ be an $n \times (n-1)$ matrix whose columns are an orthonormal basis of $\vspan\{\vec x\}^\perp$. If $\vec b \perp \vec x$, the same argument used in the proof of \cref{eqn:projected-krylov-orthogonal} gives us 
\begin{equation}
\label{eqn:projected-krylov-implicit}
f(A) \vec b = Q f(Q^TAQ) Q^T \vec b.
\end{equation}

Recall that the usual rational Arnoldi algorithm for $f(A)\vec b$ computes an orthonormal sequence $\left\{ \vec v_k \right\}_{k \ge 1}$ such that $\vec v_1 = \vec b/\norm{\vec b}_2$ and $\vspan\left\{\vec v_1, \dots, \vec v_k\right\} = \rat_k(A, \vec b)$. If we define $V_k = \left[\vec v_1\, \dots \, \vec v_k\right]$, and $B_k = V_k^T A V_k$, a rational Krylov method then yields the approximation
\begin{equation}
\label{eqn:krylov-approx--no-proj}
\bar {\vec y}_k := V_k f(B_k) V_k^T \vec b.
\end{equation}

Alternatively, if we work with the right hand side of \cref{eqn:projected-krylov-implicit}, after $k$ iterations the rational Arnoldi algorithm constructs the matrix $U_k = \left[\vec u_1 \, \dots \, \vec u_k\right] \in \R^{(n-1) \times k}$, whose columns $\{ \vec u_1, \dots, \vec u_k \}$ are an orthonormal basis for $\rat_k(Q^T A Q, Q^T \vec b)$. Then the vector $f(Q^T A Q) Q^T \vec b$ can be approximated by 
\begin{equation*}
U_k f(C_k) U_k^T Q^T \vec b,
\end{equation*}
where $C_k = U_k^T (Q^TAQ) U_k$. Applying now \cref{eqn:projected-krylov-implicit}, we have the following approximation to $f(A)\vec b$: 
\begin{equation}
\label{eqn:krylov-approx--proj}
\barbarvec y_k = Q U_k f(C_k) U_k^T Q^T \vec b.
\end{equation}
We will refer to the method described by equation \cref{eqn:krylov-approx--proj} as a \textit{projected rational Krylov method}.

The main result of this section is the following. 

\begin{theorem}
\label{thm:projected-krylov-equiv}
Let $\barvec y_k$ and $\barbarvec y_k$ be the approximations to $f(A) \vec b$ defined in \cref{eqn:krylov-approx--no-proj} and \cref{eqn:krylov-approx--proj}, respectively, where $\vec b \perp \vec x$. Then, in exact arithmetic it holds that $\barvec y_k = \barbarvec y_k$.
\end{theorem}

We start by proving a few basic properties that we will use repeatedly in the following discussion. 

\begin{lemma}
\label{lemma:projected-kylov-properties}
Using the same notation as above, the following properties hold:
\begin{enumerate}[label=\normalfont{(\alph*)}]
\item
$\rat_k(A, \vec b) \perp \vec x$ and $QQ^T \vec v = \vec v$ for all $\vec v \in \rat_k(A, \vec b)$.
\item
$(Q^T A Q - \xi I)^{-1} = Q^T (A - \xi I)^{-1} Q$ for any $\xi \in \C$ such that the inverses exist.
\item
$\rat_k(Q^T A Q, Q^T \vec b) = Q^T \rat_k(A, \vec b)$ and $ Q \rat_k(Q^T A Q, Q^T \vec b) = \rat_k(A, \vec b)$.
\item
Let $\vec w_k \in \rat_k(A, \vec b)$ and $\vec z_k = Q^T \vec w_k \in \rat_k(Q^TAQ, Q^T \vec b)$. It holds
\begin{equation*}
\hspace{-3mm}
(A - \xi_k I)^{-1}A \vec w_k \!\in\! \rat_{k}(A, \vec b) \!\iff\! (Q^TAQ - \xi_k I)^{-1}Q^TAQ \vec z_k \!\in\! \rat_{k}(Q^TAQ,Q^T \vec b).
\end{equation*}
This implies that $\vec w_k$ is a continuation vector for $\rat_k(A, \vec b)$ if and only if $\vec z_k$ is a continuation vector for $\rat_k(Q^T A Q, Q^T \vec b)$.
\end{enumerate}
\end{lemma}
\begin{proof}
(a) Let $\vec v \in \rat_k(A, \vec b)$, i.e. $\vec v = q_{k-1}(A)^{-1}p(A) \vec b$. Since it holds that $\vec x^T A = \lambda \vec x^T$, we also have $\vec x^T \vec v = \vec x^T q_{k-1}(A)^{-1}p(A) \vec b = q_{k-1}(\lambda)^{-1}p(\lambda) \vec x^T \vec b = 0$, since $\vec b \perp \vec x$. For the second part of the statement, we have $QQ^T \vec v = \vec v - \vec x \vec x^T\vec v = \vec v$ since $\vec v \perp \vec x$.

(b) Let us show that $(Q^TAQ - \xi_k I)Q^T (A - \xi_k I)^{-1}Q = I$. We have:

\begin{align*}
(Q^TAQ - \xi_k I)Q^T (A - \xi_k I)^{-1} Q &= Q^T(A-\xi_k I) QQ^T (A-\xi_k I)^{-1}Q \\
&= Q^T(A-\xi_k I) (I - \vec x \vec x^T) (A - \xi_k I)^{-1}Q \\
&= I - Q^T(A-\xi_k I) \vec x (\lambda-\xi_k)^{-1} \vec x^T Q = I,
\end{align*}
where in the last two equalities we used $\vec x^T(A - \xi_k)^{-1} = (\lambda - \xi_k)^{-1} \vec x^T$ and $\vec x^T Q = \zero^T$.

(c) Let $\vec v \in \rat_k(Q^T A Q, Q^T \vec b)$, so $\vec v = q_{k-1}(Q^T A Q)^{-1}p(Q^TAQ) Q^T \vec b$, for some polynomial $p$ with $\deg p \le k-1$. Using that $QQ^T = I - \vec x \vec x^T$ and $\vec b \perp \vec x$, we can prove that $p(Q^T A Q) Q^T \vec b = Q^T p(A) \vec b$. Because of (b), we also have
\begin{align*}
(Q^T A Q - \xi_{k-1} I)^{-1} Q^T p(A) \vec b &= Q^T (A - \xi_{k-1}I)^{-1} (I - \vec x \vec x^T) p(A) \vec b \\
&= Q^T (A - \xi_{k-1}I)^{-1} p(A) \vec b,
\end{align*}
since $x^T p(A) \vec b = p(\lambda) \vec x^T \vec b = 0$. Using similar operations on the other factors of $q_{k-1}$, we finally obtain
\begin{equation*}
\vec v = q_{k-1}(Q^T A Q)^{-1}p(Q^TAQ) Q^T \vec b = Q^T q_{k-1}(A)^{-1} p(A) \vec b \in Q^T \rat_k(A, \vec b).
\end{equation*}
The second identity follows from the fact that $QQ^T \vec v = \vec v$ for all $\vec v \in \rat_k(A, \vec b)$.

(d) Using (b) and $\rat_k(A, \vec b) \perp \vec x$, we have the following:
\begin{equation*}
(Q^TAQ - \xi_k I)^{-1}Q^TAQ \vec z_k = Q^T(A - \xi_k I)^{-1} Q Q^T A Q Q^T \vec w_k = Q^T (A - \xi_k I)^{-1} A \vec w_k.
\end{equation*}
Because of (c), it follows that $(A - \xi_k I)^{-1} A \vec w_k \in \rat_{k}(A, \vec b)$ if and only if $Q^T (A - \xi_k I)^{-1} A \vec w_k \in Q^T \rat_{k}(A, \vec b) = \rat_k(Q^T A Q, Q^T \vec b)$. The second statment follows from the fact that $(A - \xi_kI)^{-1} A \vec v \in \rat_k(A, \vec b)$ for all $\vec v \in \rat_k(A, \vec b)$, recalling that $\vec v \in \rat_k(A, \vec b)$ is a continuation vector if and only if $(A-\xi_k I)^{-1}A \vec v \in \rat_{k+1}(A, \vec b) \setminus \rat_k(A, \vec b)$.
\end{proof}

Having established some basic facts about the projected Krylov subspace, we are now ready to prove \cref{thm:projected-krylov-equiv}.

\begin{proof}[Proof of \cref{thm:projected-krylov-equiv}]
We start by showing that there exists a $k \times k$ diagonal and orthogonal matrix $D_k$ such that $U_k = Q^T V_k D_k$. Since $\vec u_1 = Q^T \vec v_1$ by definition, it is enough to prove that $\vec u_{k+1} = \pm Q^T \vec v_{k+1}$, with the assumption that $U_{k} = Q^T V_{k} D_k$, for $k \ge 1$. If we denote by $\vec z_k$ a continuation vector for $\rat_k(Q^T A Q, Q^T \vec b)$, then by \cref{lemma:projected-kylov-properties}(d) we have that $\vec w_k = Q \vec z_k$ is a continuation vector for $\rat_k(A, \vec b)$.
We have the following chain of equalities, using \cref{lemma:projected-kylov-properties}(b) for the second equality:
\begin{align*}
\vspan \left\{U_k,\, \vec u_{k+1}\right\} &= \vspan\left\{U_{k},\, (Q^TAQ - \xi_k I)^{-1} Q^TAQ \vec z_k\right\} \\
 &= \vspan\left\{Q^T V_{k} D_k,\, Q^T(A-\xi_kI)^{-1} A \vec w_k \right\} \\
 &= Q^T \vspan\left\{ V_{k},\, (A-\xi_kI)^{-1} A \vec w_k \right\} \\
 &= Q^T \vspan\left\{ V_{k},\, \vec v_{k+1} \right\} = \vspan\left\{U_k,\, Q^T \vec v_{k+1}  \right\}.
\end{align*}
The vector $Q^T \vec v_{k+1}$ is orthogonal to the columns of $U_k = Q^T V_k D_k$, since 
\begin{equation*}
\vec v_{k+1}^TQQ^TV_{k} = \vec v_{k+1}^T (I - \vec x \vec x^T) V_{k} = \zero^T.
\end{equation*}
So we have $\vec u_{k+1} = \pm Q^T \vec v_{k+1}$, since both vectors are in $\rat_k(Q^TAQ, Q^T \vec b)$ and they are orthogonal to the columns of $U_k$. Hence we have obtained $U_k = Q^TV_k D_k$ and $Q U_k = QQ^T V_k D_k = V_k D_k$, where $D_k$ is a diagonal matrix whose diagonal elements are equal to $\pm 1$.

Now it is straightforward to see that $C_k = D_k B_k D_k$ and that $\barvec y_k = \barbarvec y_k$. Indeed we have 
\begin{displaymath}
\begin{aligned}
\barbarvec y_k &= Q U_k f(C_k) U_k^T Q^T \vec b \\
&= V_k D_k f(U_k^T Q^T AQ U_k)D_k V_k^T \vec b \\
&= V_k D_k f(D_k V_k^T A V_k D_k) D_k V_k^T \vec b \\
&= V_k f(B_k) V_k^T \vec b = \barvec y_k.
\end{aligned}
\end{displaymath}
\end{proof}

The result of~\cref{thm:projected-krylov-equiv} can also be seen as a consequence of the implicit Q theorem for rational Arnoldi decompositions~\cite[Theorem~3.2]{BerljafaGuettel15}.

Although~\cref{thm:projected-krylov-equiv} is guaranteed to hold only in exact arithmetic, we observed from our experiments that the error curves given by the two approximations \cref{eqn:krylov-approx--no-proj} and \cref{eqn:krylov-approx--proj} are almost always overlapping, so the implicit projection is a valid (and cheaper) alternative to~\cref{eqn:krylov-approx--proj}.

\begin{remark}
\label{rem:approaches-similar-to-implicit-projection}
The approach described in this section is similar to the deflation and augmentation strategies used in the solution of linear systems with Krylov methods: the aim of these techniques is to include exact or approximate spectral information on the matrix in order to speed up the convergence. This can be done by either adding a few known eigenvectors to the Krylov subspace, or by directly solving a deflated problem constructed using the spectral information on the matrix. Since the implicit projection method constructs a Krylov subspace that is orthogonal to the vector $\vec x$ (see~\cref{lemma:projected-kylov-properties}(a)), it can be interpreted as an implicit way to construct an augmented Krylov subspace, also containing the eigenvector $\vec x$. For additional details on deflation and augmentation techniques used in Krylov methods for the solution of linear systems, we refer to the review article~\cite[Section~9]{SimonciniSzyld06} and to the references cited therein.
\end{remark}

%~ \section{Algorithm}
%~ \label{sec:alg}
%~ Our analysis leads to the algorithm in \cref{alg:buildtree}.
%~ \begin{algorithm}
%~ \caption{Build tree}
%~ \label{alg:buildtree}
%~ \begin{algorithmic}
%~ \STATE{Define $P:=T:=\{ \{1\},\ldots,\{d\}$\}}
%~ \WHILE{$\#P > 1$}
%~ \STATE{Choose $C^\prime\in\mathcal{C}_p(P)$ with $C^\prime := \operatorname{argmin}_{C\in\mathcal{C}_p(P)} \varrho(C)$}
%~ \STATE{Find an optimal partition tree $T_{C^\prime}$ }
%~ \STATE{Update $P := (P{\setminus} C^\prime) \cup \{ \bigcup_{t\in C^\prime} t \}$}
%~ \STATE{Update $T := T \cup \{ \bigcup_{t\in\tau} t : \tau\in T_{C^\prime}{\setminus} \mathcal{L}(T_{C^\prime})\}$}
%~ \ENDWHILE
%~ \RETURN $T$
%~ \end{algorithmic}
%~ \end{algorithm}

\section{Numerical experiments}
\label{sec:numerical-experiments}
In this section we test and compare the performance of the various methods for the computation of $f(A) \vec b$ that we presented earlier, using them to approximate the solution to the fractional diffusion equation \cref{eqn:diffusion-fractional} on real-world networks, both undirected and directed. 
Recall that the solution to \cref{eqn:diffusion-fractional} at time $t$ can be expressed in the form
\begin{equation}
\label{eqn:experiments--solution}
\vec u(t) = f(L^T) \vec u_0, \qquad f(z) = e^{-t z^\alpha}, \quad t \ge 0, \quad \alpha \in (0,1).
\end{equation}
Since the graphs we consider are strongly connected, the eigenvalues $\lambda_1, \dots, \lambda_n$ of the graph Laplacian can be ordered in a way such that $0 = \lambda_1 < \abs{\lambda_2} \le \dots \le \abs{\lambda_n}$. 

We use the result of \cref{thm:projected-krylov-equiv} to compute $f(L^T) \vec u_0$ via \cref{eqn:projected-krylov-general} in the following way: letting $\beta = \one^T \vec u_0 > 0$, there exists $\vec w \perp \one$ such that $\vec u_0 = \vec w + \beta \vec z$ (recall that $L^T \vec z = \zero$), and thus we can compute $f(L^T) \vec u_0$ as
\begin{equation}
\label{eqn:experiments--implicit-projection}
f(L^T) \vec u_0 = f(L^T) \vec w + \beta f(L^T) \vec z = f(L^T) \vec w + \beta \vec z.
\end{equation}
\cref{thm:projected-krylov-equiv} guarantees that, at least in exact arithmetic, a rational Krylov method for $f(L^T) \vec w$ yields the same approximate solution as the same Krylov method applied to the projected problem $f(Q^T L^T Q) Q^T \vec w$. We refer to the method obtained by using~\cref{eqn:experiments--implicit-projection} and approximating $f(L^T)\vec w$ with a Krylov method as an \emph{implicitly projected Krylov method}.

As we mentioned earlier, we use poles located on the negative real axis $(-\infty, 0)$. For the Shift-and-Invert Krylov method, we compare two different choices of poles. 
Recall that in the case of a symmetric positive definite matrix $A$, if $a>0$ is a lower bound for the smallest eigenvalue of $A$ and $b>0$ is an upper bound for its largest eigenvalue, so that $\sigma(A)\subset [a, b]$, an effective pole choice for the Shift-and-Invert Krylov method is given by $\xi = - \sqrt{ab}$ (see, e.g., \cite[Section~6]{BeckermannReichel09}). In analogy with this choice, we use the pole $\xi = -\sqrt{\abs{\lambda_2 \lambda_n}}$: if the graph is undirected, when we use the rank-one shift approach \cref{eqn:rank-one-shift-krylov} with $\theta \ge \abs{\lambda_2}$, this pole corresponds exactly to the optimal choice $\xi = -\sqrt{\lambda_\text{min}\lambda_\text{max}}$ for symmetric positive definite matrices; the same choice appears to be reasonable also for the singular graph Laplacian and in the directed case. Indeed, our experiments show that this choice always provides a reliable convergence rate.

As a second possible choice for the Shift-and-Invert Krylov method, we use the pole $\xi = -t^{-2/\alpha}$ proposed in \cite{MoretNovati19}; this choice is based on an integral bound for the error of the Shift-and-Invert Krylov method, obtained using an integral expression for the function $f$. The choice $\xi = -t^{-2/\alpha}$ was proposed for specific functions that arise in the context of fractional differential equations, like $f(z) = e^{-tz^\alpha}$ or $f(z) = (1 + tz^\alpha)^{-1}$, with $\alpha \in (0,1)$ and $t > 0$. This pole choice is particularly effective when $t$ is large, but it is more sensitive to changes in the parameters.

For the rational Krylov method based on the equidistributed sequence (EDS) described in \cref{sec:laplace-stieltjes-functions}, we computed the asymptotically optimal poles using the spectral interval $[0.99 \cdot\abs{\lambda_2}, \,1.01\cdot\abs{\lambda_n}]$, once again ignoring the presence of the eigenvalue of the Laplacian at $0$. The resulting poles are located on the negative real axis $(-\infty, 0)$.
Similarly to the choice $\xi = -\sqrt{\abs{\lambda_2 \lambda_n}}$ for the Shift-and-Invert Krylov method, this pole sequence is guaranteed to have the asymptotic rate of convergence \cref{eqn:laplace-stieltjes-convergence-rate} on the rank-one shifted matrix $L^T + \theta \vec z \one^T$ when the graph is undirected (for $\theta \ge \abs{\lambda_2}$). The experiments that we performed suggest that the same choice is still very effective also in the directed case and even when applied directly to the singular graph Laplacian. Note that this method, as well as the Shift-and-Invert method with $\xi = -\sqrt{\abs{\lambda_2 \lambda_n}}$, requires the knowledge of the largest and smallest nonzero eigenvalues of the graph Laplacian $L$, that have to be computed beforehand.

All the experiments were performed in Matlab using the \texttt{rat\_krylov} function in the Rational Krylov toolbox \cite{RKToolbox}. The shifted linear systems in the Shift-and-Invert Krylov method were solved by computing beforehand an LU decomposition of the permuted matrix $P^T L^T P - \xi I$, where $P$ is a fill-in reducing permutation matrix obtained using the \texttt{amd} Matlab function.  In the rank-one shifted and in the projected version of the methods, we used the modified Sherman-Morrison formula \cref{eqn:cancellation--final}, with the vector $\vec \psi$ defined in~\cref{eqn:cancellation--phi}, and the identities \cref{eqn:projection-operations} to avoid explicitly forming the dense matrices $L^T + \theta \vec z \one^T$ and $Q^T L^T Q$. The use of \cref{eqn:cancellation--final} over \cref{eqn:Sherman-Morrison-laplacian} allowed us to avoid the cancellation in \cref{eqn:Sherman-Morrison-cancellation} for poles close to zero, as discussed after \cref{rem:Sherman-Morrison-small-xi}. We set $\theta = 1$ in \cref{eqn:rank-one-shift-krylov} and we used the matrix $Q$ defined by \cref{eqn:projection-Q}. 

The error that we display is the relative error in the $2$-norm, $\norm{\vec y - \barvec y_k}_2$, where $\vec y$ is the solution to \cref{eqn:diffusion-fractional} at a certain time $t$, or an accurate approximation computed with a Krylov method when the size of the graph is large. In all our experiments, we first extracted the largest strongly connected component (LCC) of a graph and we restricted our problem to that component. Information on the number of nodes and edges of these components, as well as the maximum and minimum nonzero eigenvalues of the corresponding graph Laplacians, are reported in \cref{table:1}. The real-world networks that we used are available in the Sparse Matrix Collection \cite{SparseMatrixCollection}.

\begin{table}
\caption{Some information on the graphs used in the experiments. $A$ denotes the $n \times n$ adjacency matrix of the LCC of each graph.}
\label{table:1}
\begin{center}
\begin{tabular}{| l || r r r c c |}
\hline
Graph & $n$ & \texttt{nnz(A)} & \texttt{nnz(A-A')} & $\lambda_2$ & $\lambda_n$ \\
\hline
\hline
minnesota & 2640 & 6604 & 0 & 8.45e-04 & 6.88e+00 \\
Oregon-1 & 11174 & 46818 & 0 & 8.44e-02 & 2.39e+03 \\
ca-HepPh & 11204 & 235238 & 0 & 3.55e-02 & 4.92e+02 \\
as-22july06 & 22963 & 96872 & 0 & 5.07e-02 & 2.39e+03 \\
Roget & 904 & 4830 & 4128 & 8.22e-02 & 2.27e+01 \\
wiki-Vote & 1300 & 39456 & 67204 & 3.73e-01 & 5.96e+02 \\
enron & 8271 & 146260 & 212124 & 7.01e-02 & 8.79e+02 \\
p2p-Gnutella30 & 8490 & 31706 & 63412 & 2.47e-01 & 2.00e+01 \\
hvdc1 & 24836 & 133620 & 4856 & 2.03e-04 & 1.80e+02 \\
\hline
\end{tabular}
\end{center}
\end{table}

As we observed in \cref{sec:fractional-laplacian}, the solution $\vec u(t)$ to the fractional diffusion equation \cref{eqn:diffusion-fractional} is a probability vector for all $t \ge 0$, and hence it is desirable for the approximations computed with Krylov methods to have the same property. In our experiments, we observed that this is indeed the case for the Krylov methods applied to the shifted matrix $L^T + \vec z \one^T$, as well as for the projected and implicitly projected Krylov methods. On the other hand, in general the approximate solutions obtained by working directly with the singular graph Laplacian $L^T$ do not have entries that sum up to 1, and often exhibit a wildly oscillating error; moreover, upon closer inspection, we observed that a significant portion of the error lied along the left null-eigenvector $\vec z$ of the graph Laplacian $L$. By subtracting a multiple of $\vec z$ from the approximate solution to enforce $\one^T \barvec y_k = 1$, we were able to ``correct'' this oscillating component of the error; specifically, for each $k$ we replaced the approximate solution $\barvec y_k$ obtained after $k$ iterations of a Krylov method with the corrected approximation
\begin{equation}
\label{eqn:experiments--correction-step}
\barvec y_k - (\one^T \barvec y_k - 1) \vec z,
\end{equation}
whose entries now sum up to 1. 
We found that this correction greatly reduced the error both in the undirected and in the directed case, and hence we
always applied it to the standard version of Krylov methods in all our experiments. On the other hand, the error correction \cref{eqn:experiments--correction-step} was never needed for the shifted, projected and implicitly projected variants of the methods. The error with and without the correction \cref{eqn:experiments--correction-step} is illustrated for the directed graph \texttt{Roget} in \cref{fig:1}.

% ERROR CORRECTION EXAMPLE
\begin{figure}
\makebox[\linewidth][c]{
\begin{subfigure}{.7\textwidth}
\centering
\includegraphics[width=\linewidth]{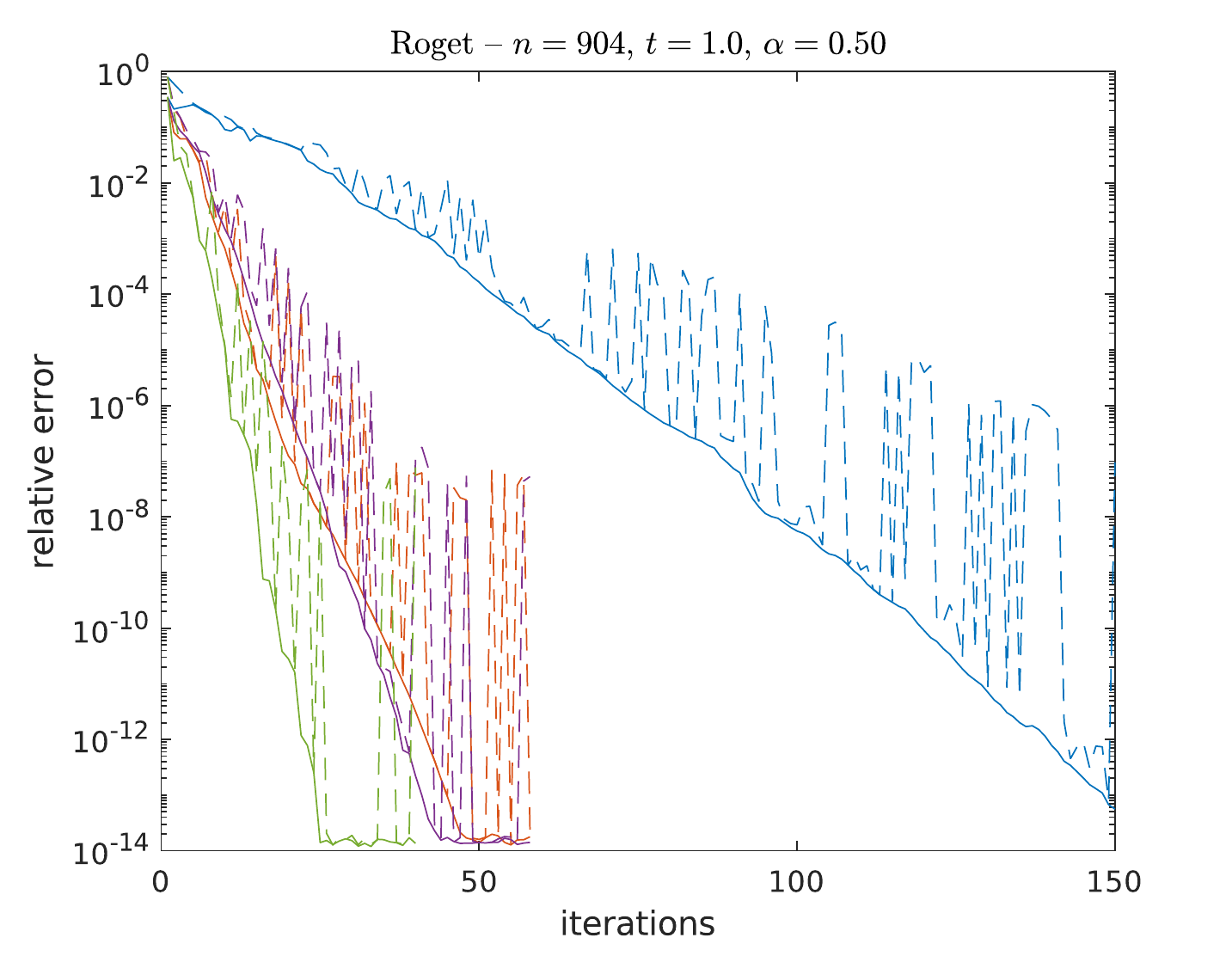}
\caption{}
\end{subfigure}
\hfill
\begin{subfigure}{.45\textwidth}
\centering
\includegraphics[width=\linewidth]{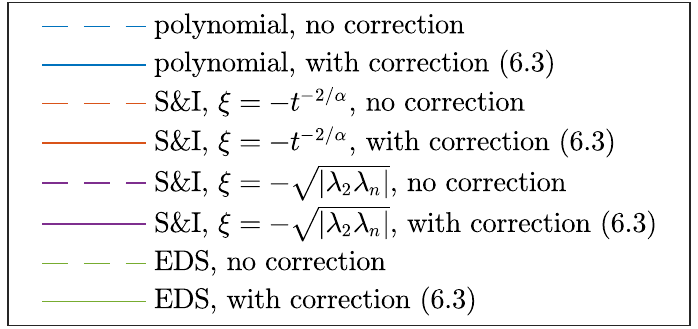}
\caption{}
\end{subfigure}}
\caption{Convergence of Krylov methods for the computation of the solution to the fractional diffusion equation \cref{eqn:diffusion-fractional} with $\alpha = 0.5$ and $t = 1$ on the LCC of the directed graph \texttt{Roget}, with and without the error correction \cref{eqn:experiments--correction-step}. Note the logarithmic scale on the vertical axis.}
\label{fig:1}
\end{figure}

% SMALL UNDIRECTED
\begin{figure}
\makebox[\linewidth][c]{
\begin{subfigure}{.6\textwidth}
\centering
\includegraphics[width=\linewidth]{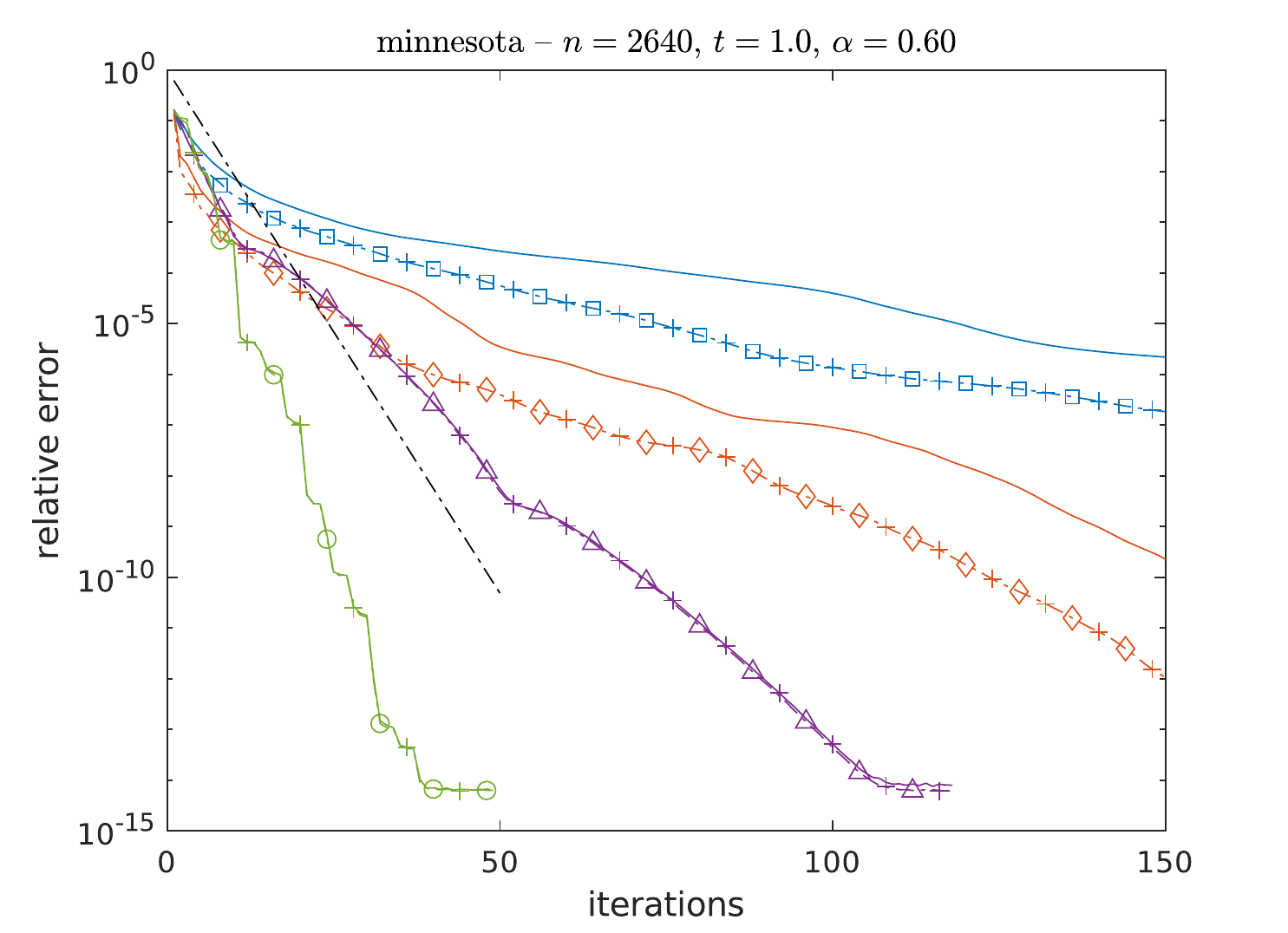}
\caption{}
\end{subfigure}
\hfill
\begin{subfigure}{.6\textwidth}
\centering
\includegraphics[width=\linewidth]{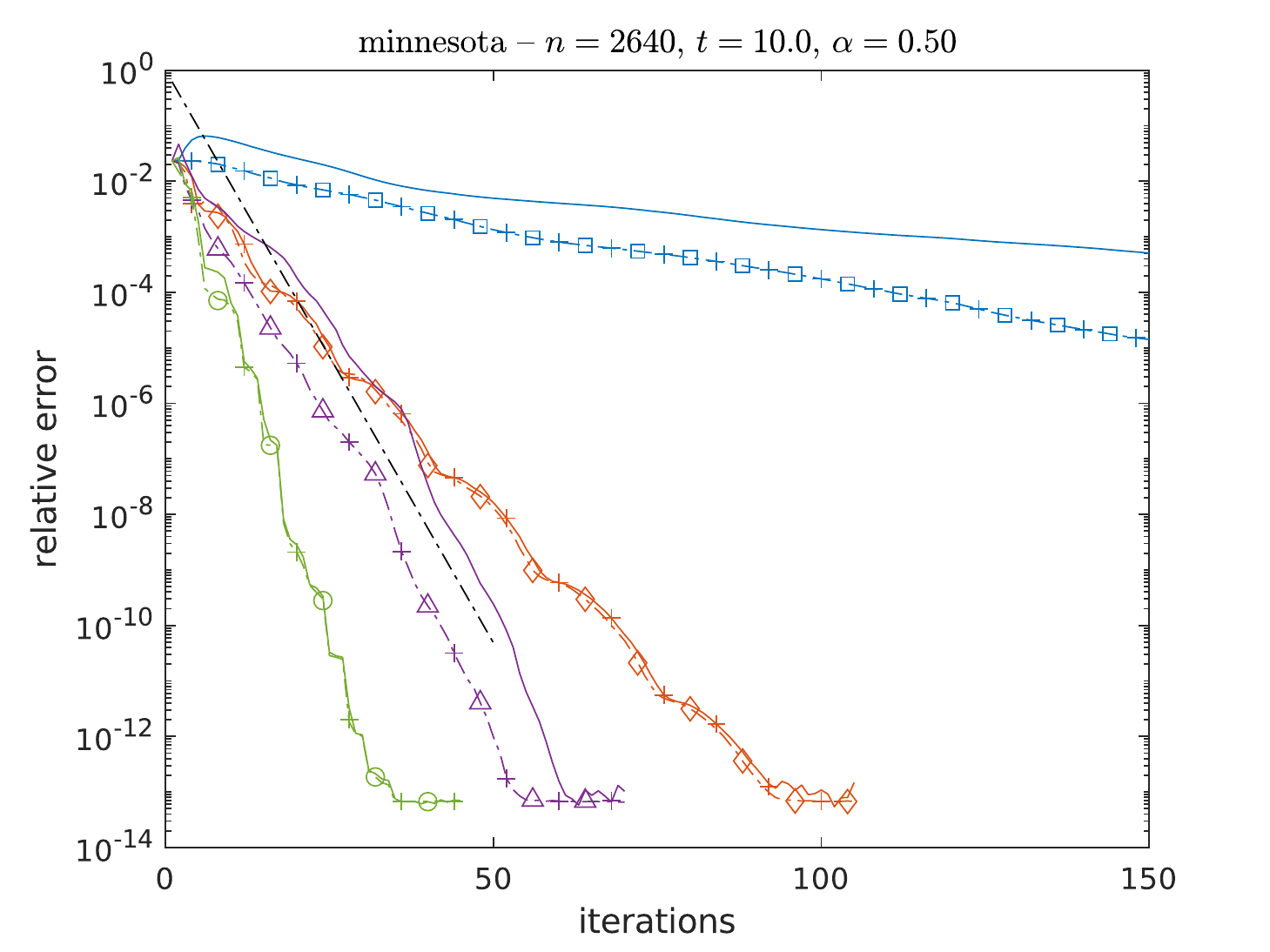}
\caption{}
\end{subfigure}}
\hfill
\makebox[\linewidth][c]{
\begin{subfigure}{.6\textwidth}
\centering
\includegraphics[width=\linewidth]{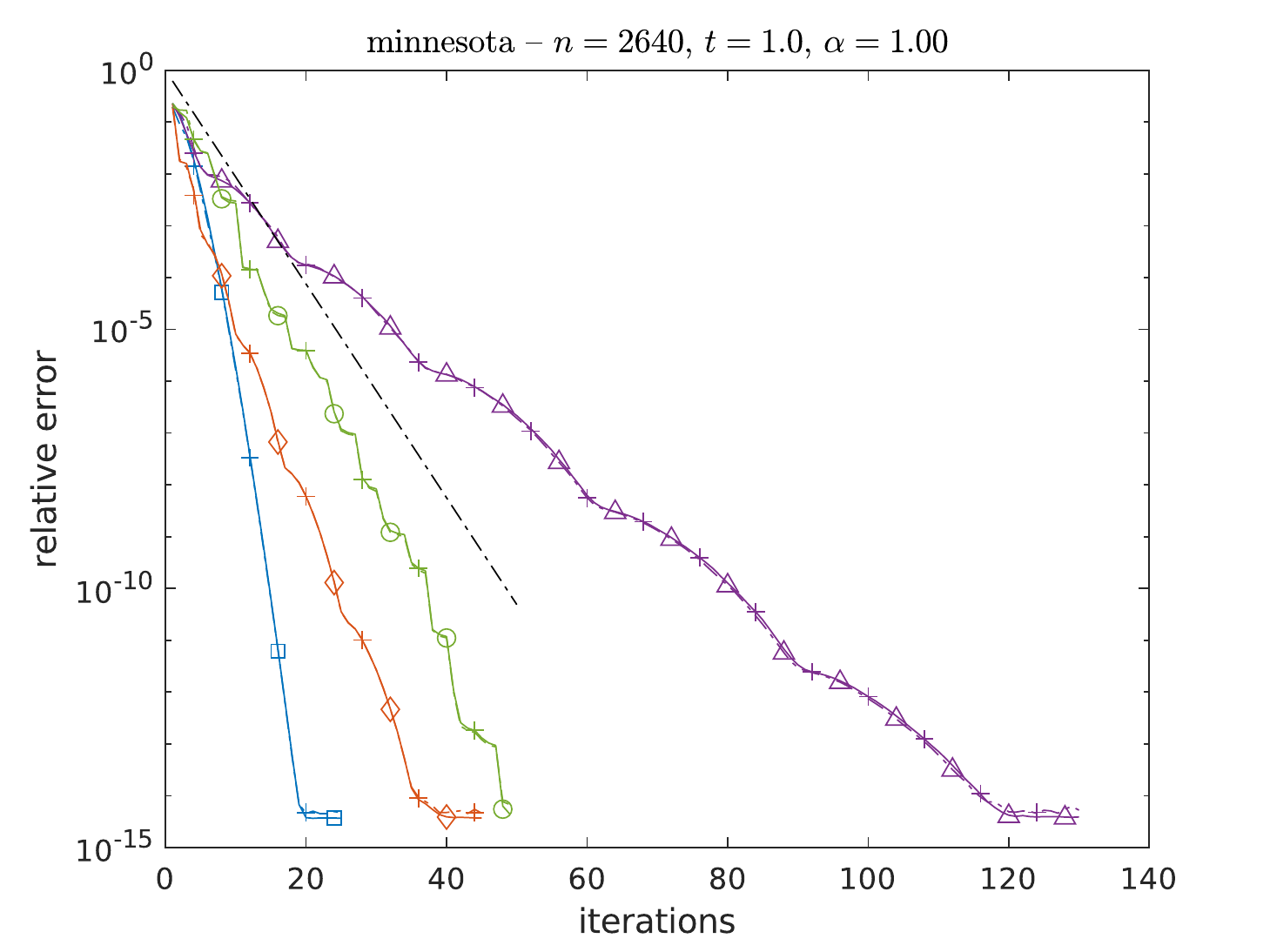}
\caption{}
\end{subfigure}
\hfill
\begin{subfigure}{.6\textwidth}
\centering
\includegraphics[width=.55\linewidth]{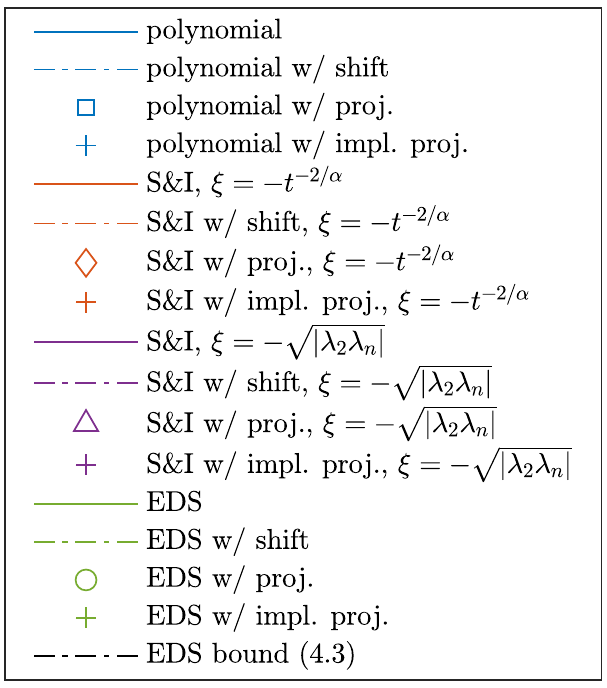}
\caption{}
\end{subfigure}}
\caption{Convergence of Krylov methods for the computation of the solution to the fractional diffusion equation \cref{eqn:diffusion-fractional} on the LCC of the undirected graph \texttt{minnesota}, for different values of $\alpha$ and $t$. Note the logarithmic scale on the vertical axis.}
\label{fig:2}
\end{figure}

% SMALL DIRECTED
\begin{figure}
\makebox[\linewidth][c]{
\begin{subfigure}{.6\textwidth}
\centering
\includegraphics[width=\linewidth]{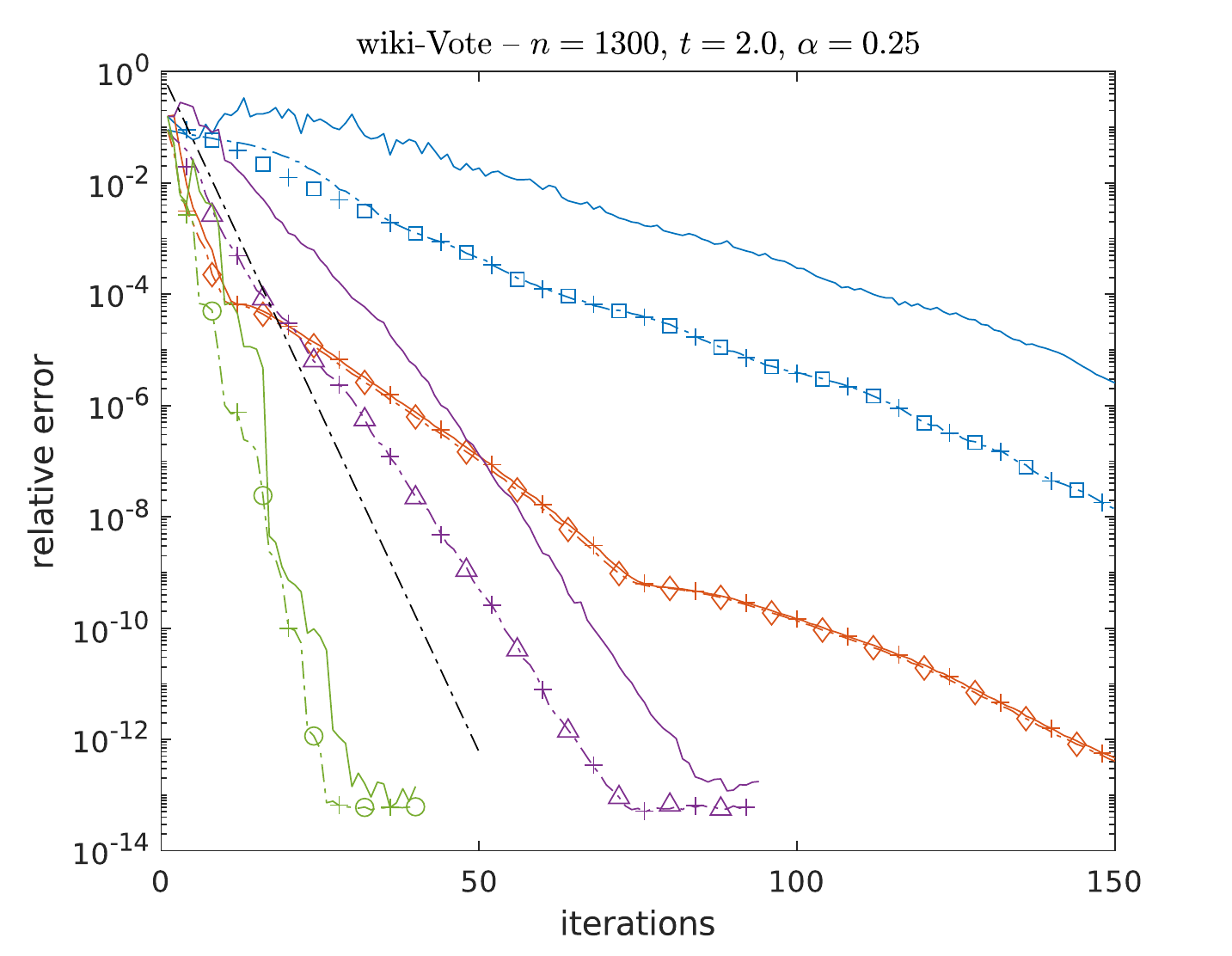}
\caption{}
\end{subfigure}
\hfill
\begin{subfigure}{.6\textwidth}
\centering
\includegraphics[width=\linewidth]{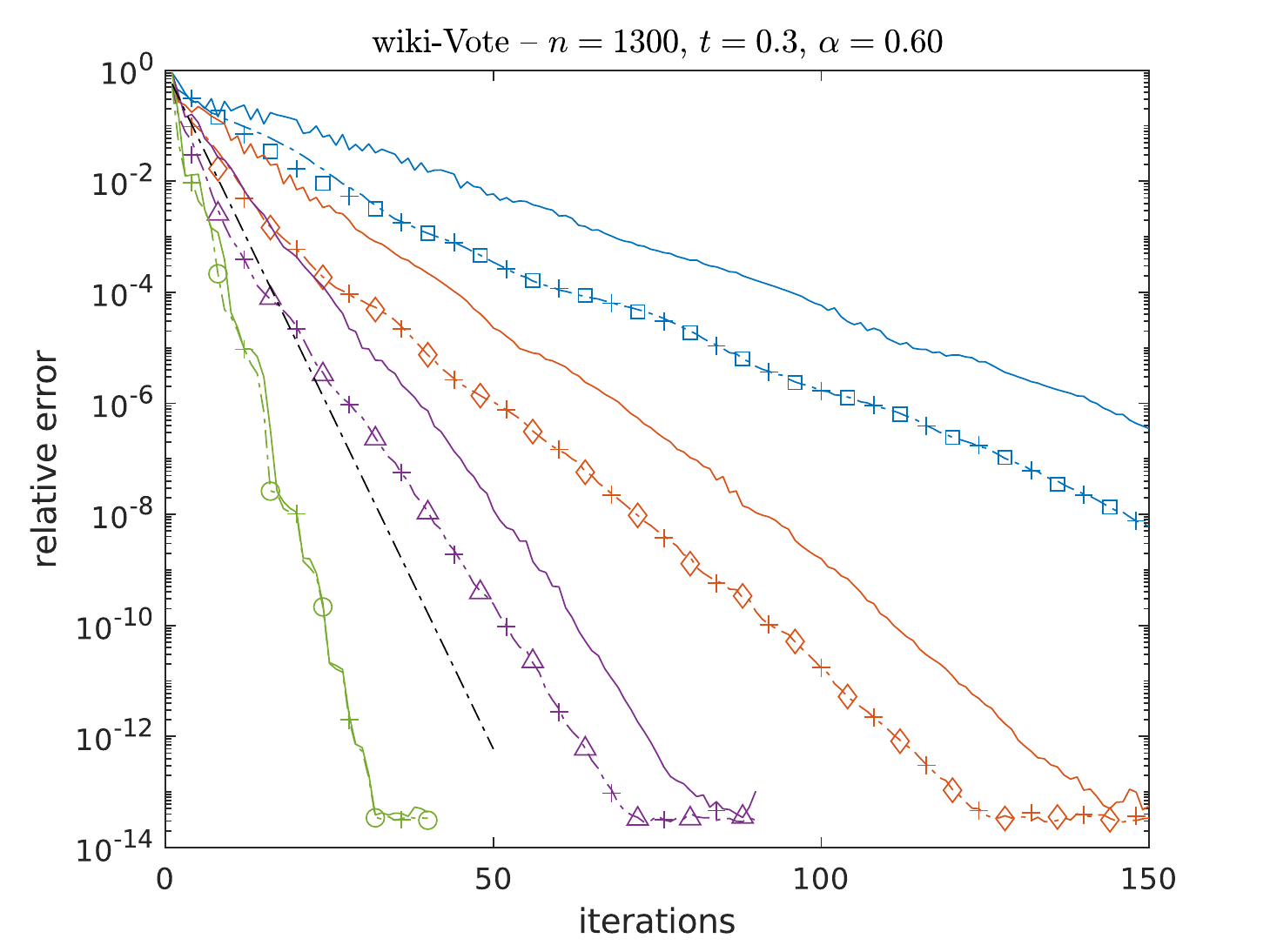}
\caption{}
\end{subfigure}}
\hfill
\makebox[\linewidth][c]{
\begin{subfigure}{.6\textwidth}
\centering
\includegraphics[width=\linewidth]{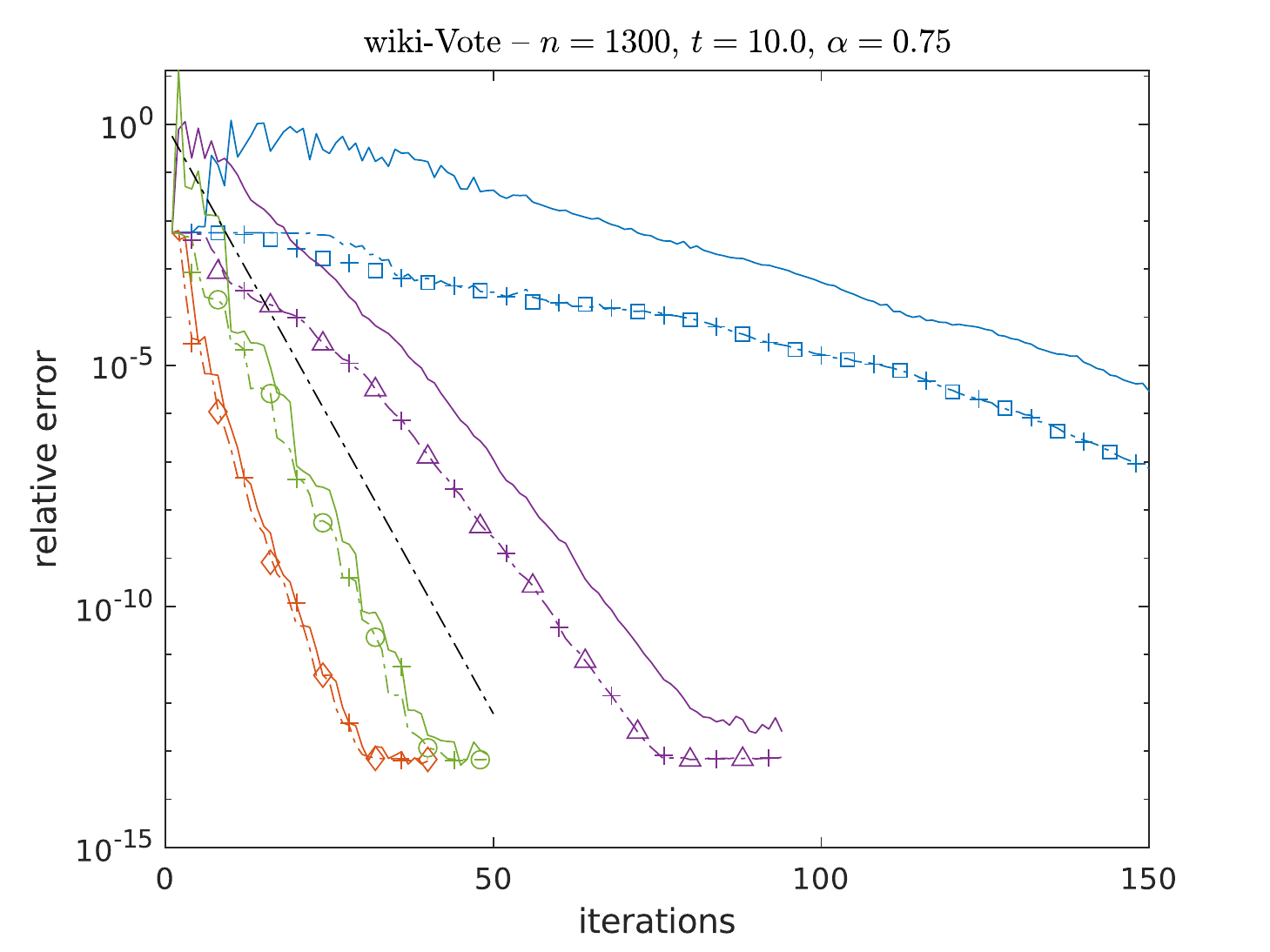}
\caption{}
\end{subfigure}
\hfill
\begin{subfigure}{.6\textwidth}
\centering
\includegraphics[width=.55\linewidth]{legend.pdf}
\caption{}
\end{subfigure}}
\caption{Convergence of Krylov methods for the computation of the solution to the fractional diffusion equation \cref{eqn:diffusion-fractional} on the LCC of the directed graph \texttt{wiki-Vote}, for different values of $\alpha$ and $t$. Note the logarithmic scale on the vertical axis.}
\label{fig:3}
\end{figure}

% LARGE UNDIRECTED
\begin{figure}
\makebox[\linewidth][c]{
\begin{subfigure}{.6\textwidth}
\centering
\includegraphics[width=\linewidth]{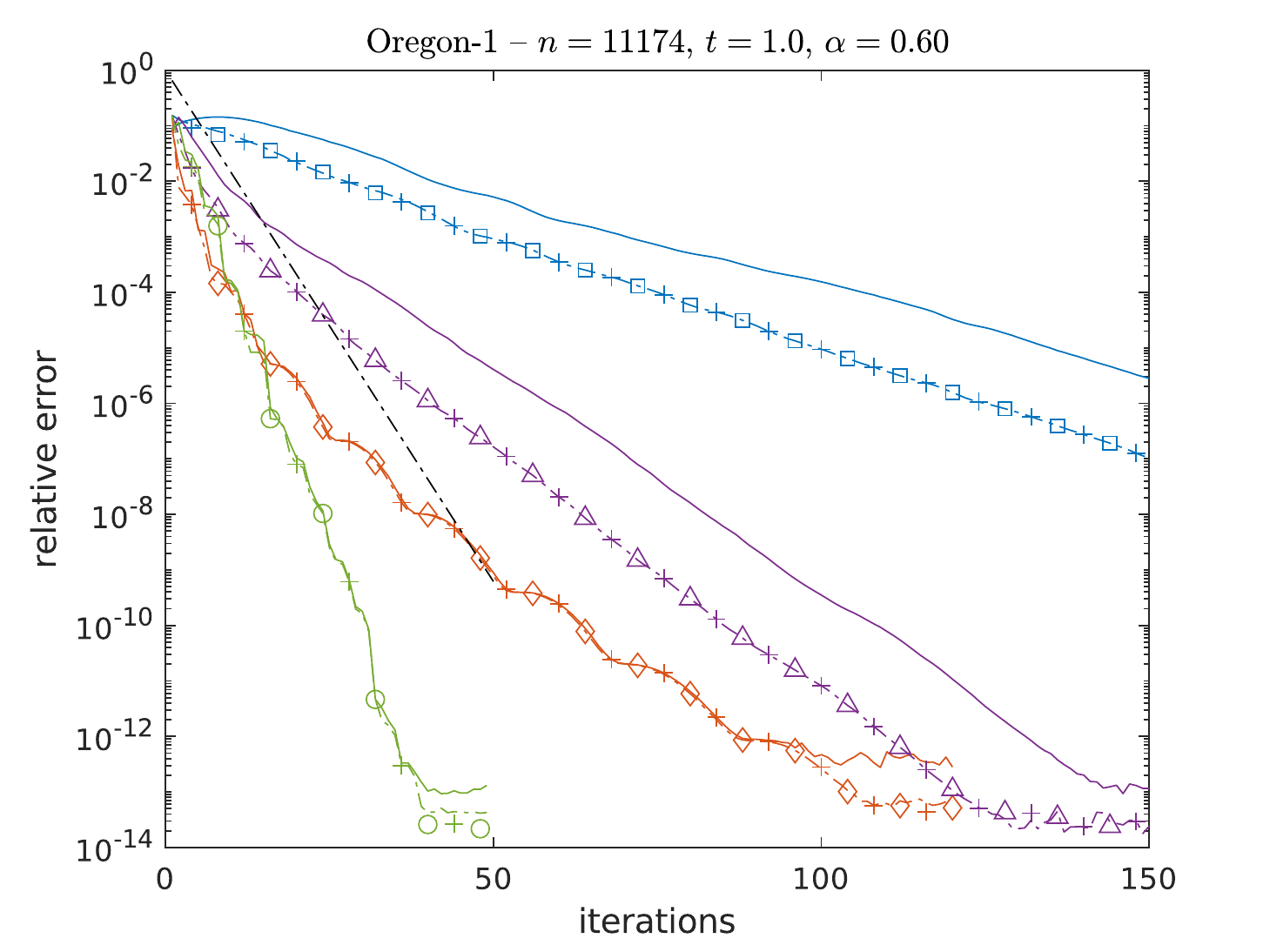}
\caption{}
\end{subfigure}
\hfill
\begin{subfigure}{.6\textwidth}
\centering
\includegraphics[width=\linewidth]{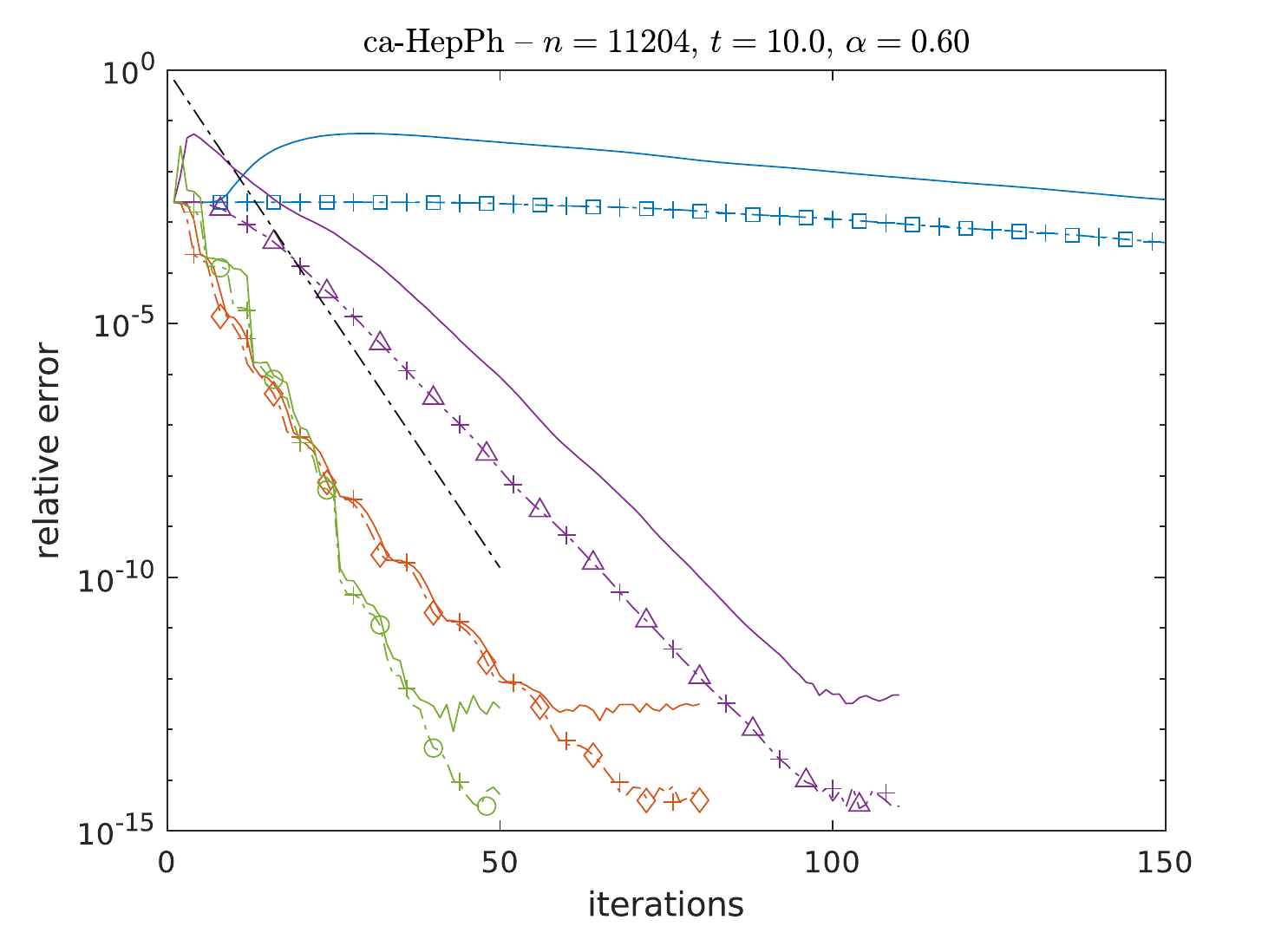}
\caption{}
\end{subfigure}}
\hfill
\makebox[\linewidth][c]{
\begin{subfigure}{.6\textwidth}
\centering
\includegraphics[width=\linewidth]{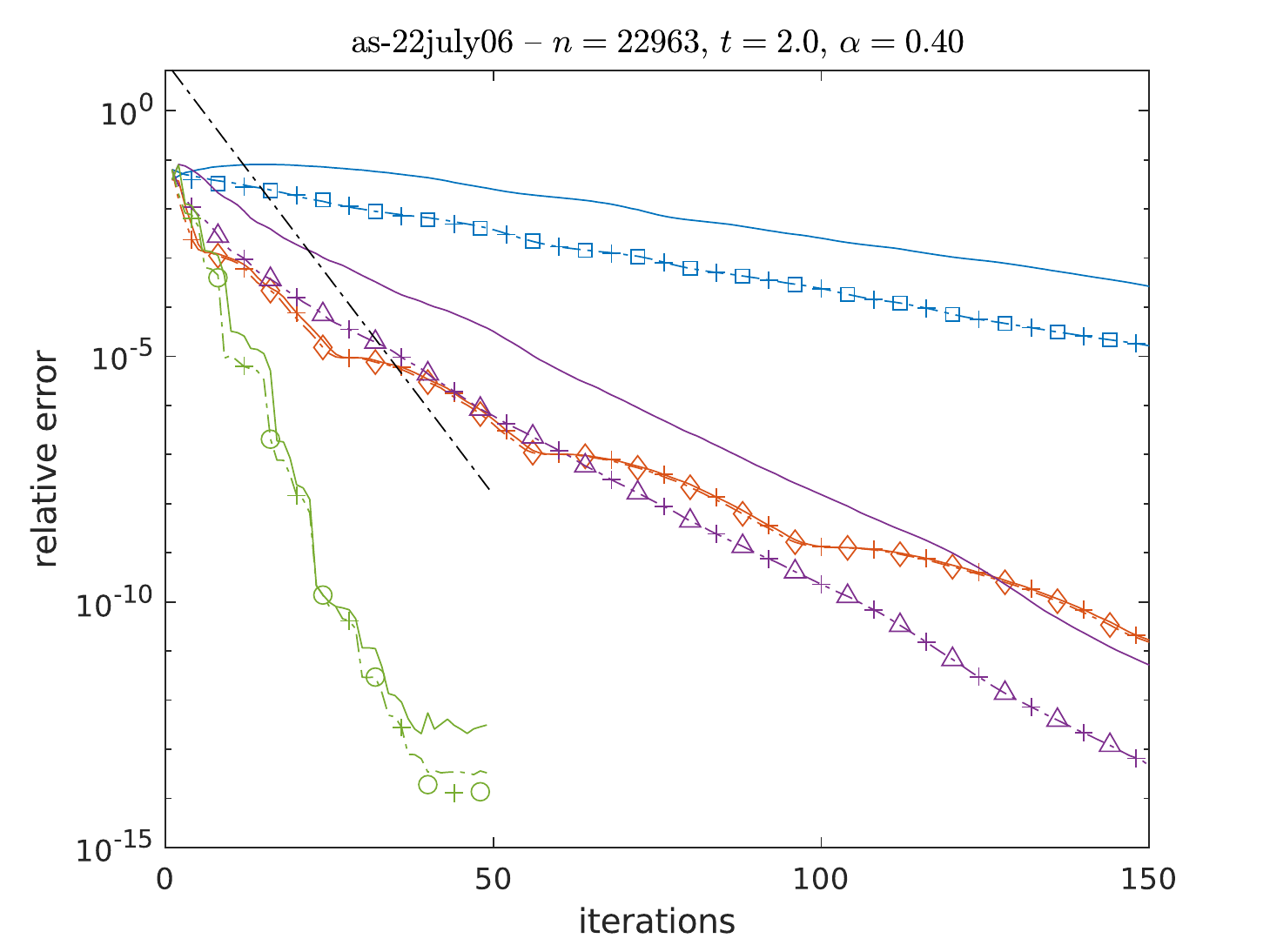}
\caption{}
\end{subfigure}
\hfill
\begin{subfigure}{.6\textwidth}
\centering
\includegraphics[width=.55\linewidth]{legend.pdf}
\caption{}
\end{subfigure}}
\caption{Convergence of Krylov methods for the computation of the solution to the fractional diffusion equation \cref{eqn:diffusion-fractional} on the LCC of the undirected graphs \texttt{Oregon-1}, \texttt{ca-HepPh} and \texttt{as-22july06}, for different values of $\alpha$ and $t$. Note the logarithmic scale on the vertical axis.}
\label{fig:4}
\end{figure}

% LARGE DIRECTED
\begin{figure}
\makebox[\linewidth][c]{
\begin{subfigure}{.6\textwidth}
\centering
\includegraphics[width=\linewidth]{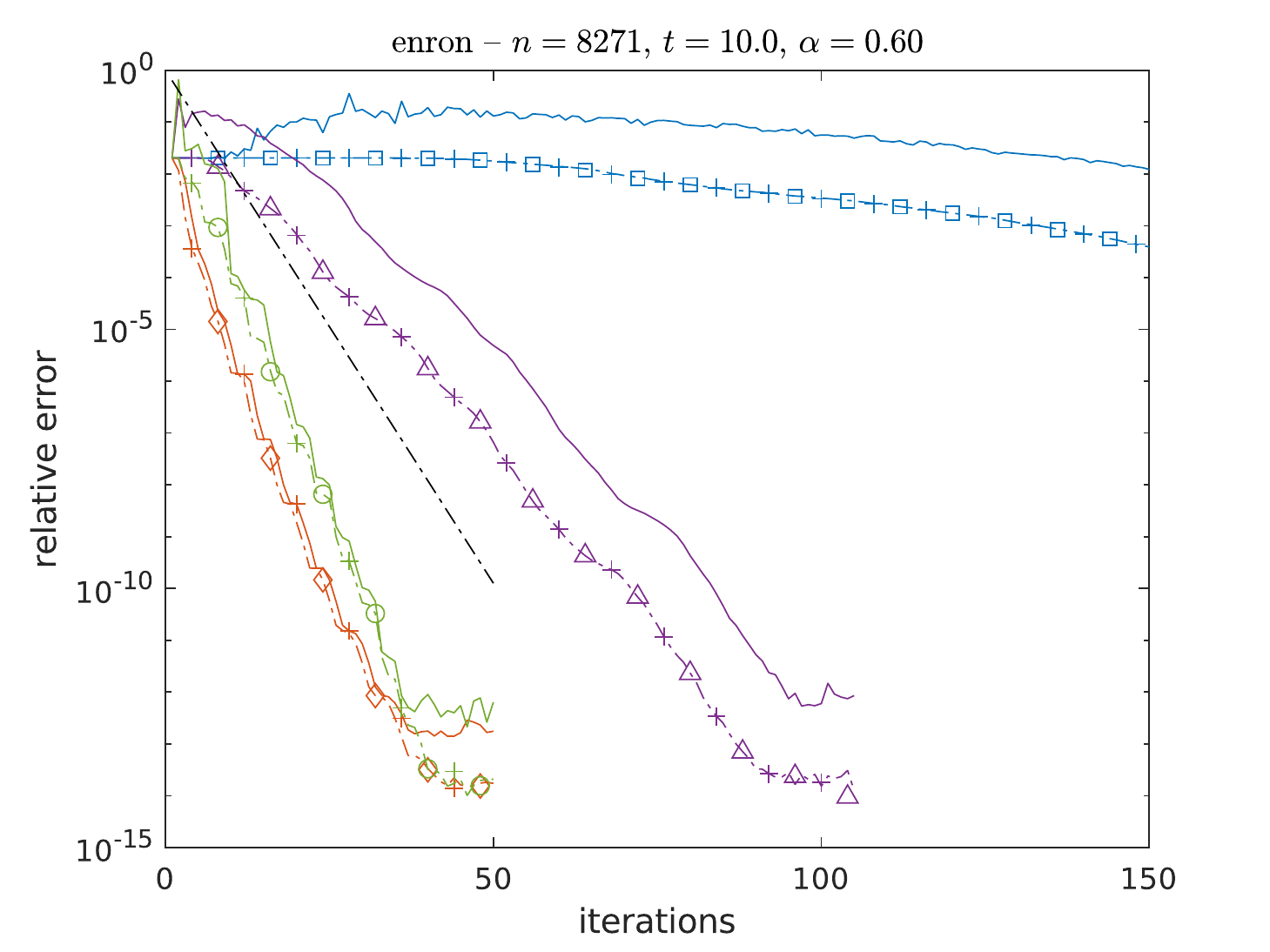}
\caption{}
\end{subfigure}
\hfill
\begin{subfigure}{.6\textwidth}
\centering
\includegraphics[width=\linewidth]{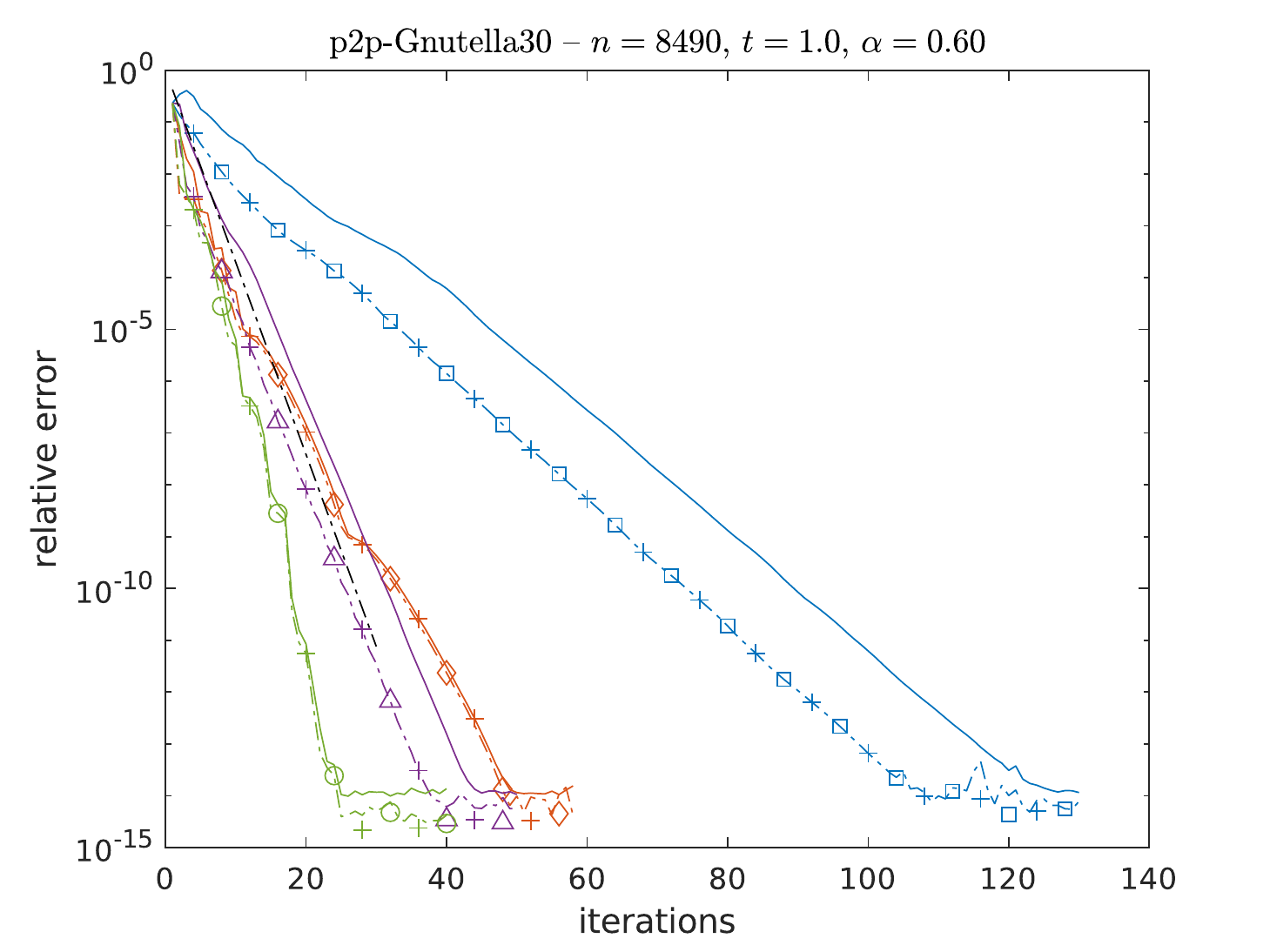}
\caption{}
\end{subfigure}}
\hfill
\makebox[\linewidth][c]{
\begin{subfigure}{.6\textwidth}
\centering
\includegraphics[width=\linewidth]{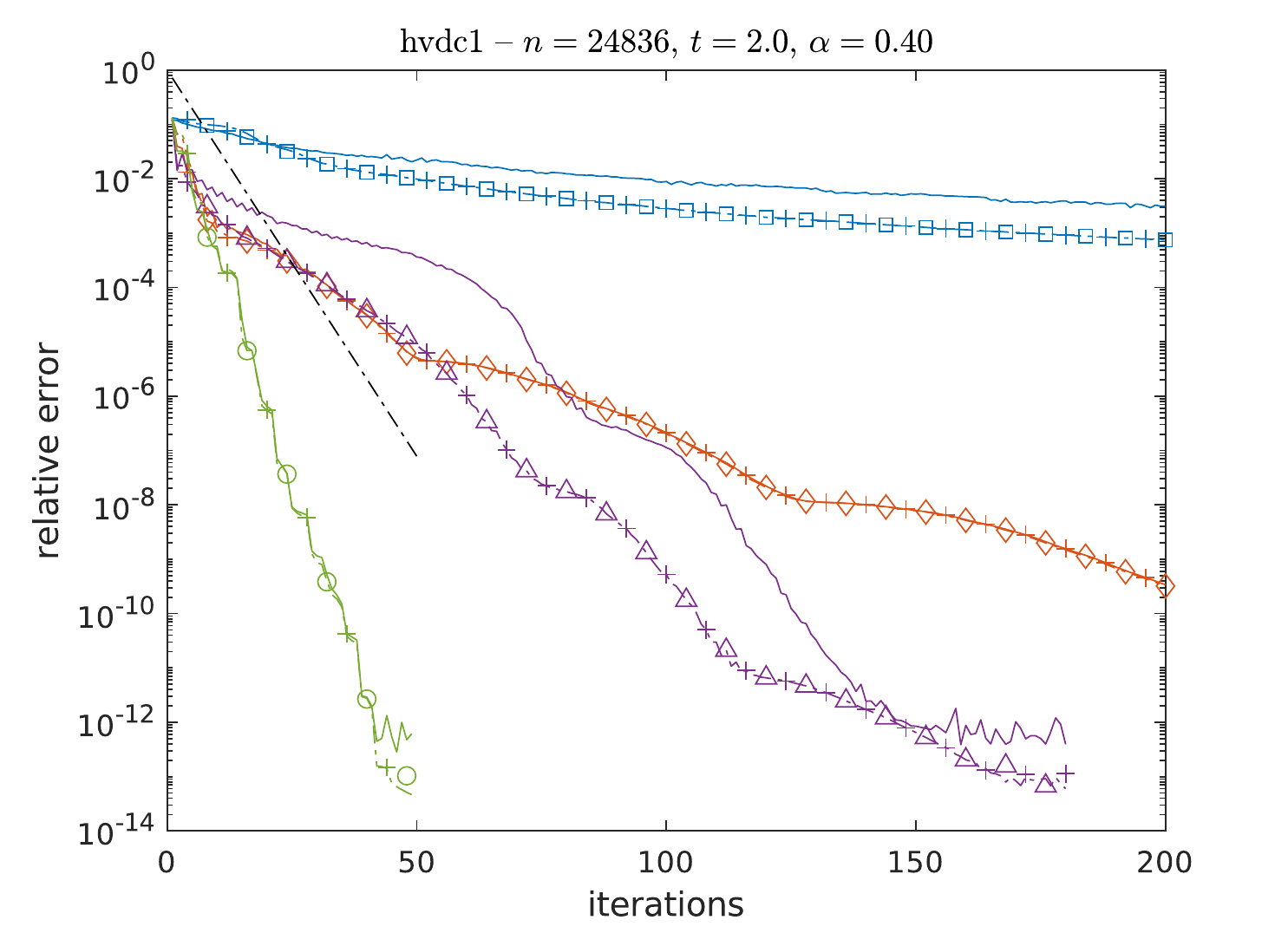}
\caption{}
\end{subfigure}
\hfill
\begin{subfigure}{.6\textwidth}
\centering
\includegraphics[width=.55\linewidth]{legend.pdf}
\caption{}
\end{subfigure}}
\caption{Convergence of Krylov methods for the computation of the solution to the fractional diffusion equation \cref{eqn:diffusion-fractional} on the LCC of the directed graphs \texttt{enron}, \texttt{p2p-Gnutella} and \texttt{hvdc1}, for different values of $\alpha$ and $t$. Note the logarithmic scale on the vertical axis.}
\label{fig:5}
\end{figure}

The first set of experiments is performed on graphs of moderate size (about $1000$ nodes), where the solution to \cref{eqn:experiments--solution} can still be computed directly in a reasonable amount of time via an eigendecomposition of the graph Laplacian. In these experiments, we used the largest connected component (LCC) of the undirected graph \texttt{minnesota} (2640 nodes) and the LCC ot the directed graph \texttt{wiki-Vote} (1300 nodes). The results for different values of $t$ and $\alpha$ are shown in \cref{fig:2,fig:3}. 
We can see that the EDS method always converges in the smallest number of iterations, with a rate that is always equal to or better than the one predicted by the bound~\cref{eqn:laplace-stieltjes-convergence-rate}, even in the nonsymmetric case. The Shift-and-Invert (S\&I) Krylov method with $\xi = -\sqrt{\abs{\lambda_2 \lambda_n}}$ has a reliable convergence rate for all choices of the parameters, while the one with $\xi = -t^{-2/\alpha}$ is more effective for large $t$ (see, e.g., \cref{fig:3}(c)); however, it is more sensitive to changes in the parameters, and sometimes converges slowly (\cref{fig:2}(a)). As expected, the polynomial Krylov method usually converges very slowly, except for the case $\alpha = 1$, corresponding to the matrix exponential (\cref{fig:2}(c)), for which polynomial Krylov methods are known to be effective. Note that in~\cref{fig:2}(b) it holds $t^{-2/\alpha} = 10^{-4}$, and observe that the rank-one shifted S\&I method with $\xi = -t^{-2/\alpha}$ attains the same accuracy as the other methods, as a result of formula \cref{eqn:cancellation--final}; because of the cancellation discussed in~\cref{rem:Sherman-Morrison-small-xi}, the same accuracy could not be attained by using \cref{eqn:Sherman-Morrison-laplacian} in place of \cref{eqn:cancellation--final}.
The error curves of the desingularized methods are always overlapped to each other (showing, in particular, that the result of~\cref{thm:projected-krylov-equiv} also holds in finite precision arithmetic), and they often represent an improvement over the standard version of Krylov methods. Note that the desingularization techniques seem to be always effective for the polynomial Krylov method, reducing the error of at least one or two orders of magnitude compared to the standard version. We also point out that in certain cases the desingularized methods manage to attain a higher final accuracy: this is most apparent in \cref{fig:3}(c) for the S\&I method with $\xi = -\sqrt{\abs{\lambda_2\lambda_n}}$.

The second set of experiments deals with larger graphs with about $10000$ or $20000$ nodes, for which the computation of an eigendecomposition of the graph Laplacian would be very expensive. Based on the results of the experiments on smaller graphs, in this case we compute the error using as the reference solution an approximation to $f(L^T) \vec u_0$ computed using the EDS rational Krylov method with implicit projection, stopping the iterations when the $2$-norm of the difference between two consecutive iterates is of the order of machine precision. To avoid bias in the error curves, we chose a different starting point for the EDS of the reference solution, thus producing a sequence of poles that is different from the one used to plot the error of the EDS method. We used the LCC of the undirected graphs \texttt{Oregon-1} (11174 nodes), \texttt{ca-HepPh} (11204 nodes) and \texttt{as-july06} (22963 nodes), and of the directed graphs \texttt{enron} (8271 nodes), \texttt{p2p-Gnutella30} (8490 nodes) and \texttt{hvdc1} (24836 nodes). The results for different values of $\alpha$ and $t$ are shown in \cref{fig:4,fig:5}.
The use of desingularization is again shown to be beneficial, often leading to faster convergence and attaining a better maximum accuracy. In \cref{fig:4}(b) and \cref{fig:5}(a) we can observe that the S\&I method with $\xi = -t^{-2/\alpha}$ does not suffer from cancellation by virtue of~\cref{eqn:cancellation--final}, despite the presence of poles close to zero. These experiments also show that the polynomial Krylov method can have a variable and unpredictable convergence rate, depending on the graph: there are situations in which the convergence can take place quickly (\cref{fig:5}(b)) or with moderate speed (\cref{fig:4}(a)), but more often than not this method converges very slowly and the error practically stagnates (see, e.g., \cref{fig:4}(b)).

%~ \Cref{fig:testfig} shows some example results. 
%~ \begin{figure}[htbp]
  %~ \centering
  %~ \label{fig:a}\includegraphics{lexample_fig1}
  %~ \caption{Example figure using external image files.}
  %~ \label{fig:testfig}
%~ \end{figure}

%~ \section{Discussion of \texorpdfstring{{\boldmath$Z=X \cup Y$}}{Z = X union Y}}

\section{Conclusions}
\label{sec:conclusions}
In this work we have discussed the use of rational Krylov methods for the solution of the fractional diffusion equation on a graph. In order to improve the convergence speed of the methods, we have proposed three different procedures to deal with the eigenvalue at zero of the graph Laplacian, namely a rank-one shift, a subspace projection, and an implicit version of this projection. The experiments we conducted show that these three procedures yield in practice the same convergence curves, and often they are faster and attain higher accuracy than the original Krylov methods. To be applied, these methods only require the computation of the left zero-eigenvector of the graph Laplacian, and an additional cost of $O(n)$ per iteration for the rank-one shift and projection techniques. The implicit projection approach is extremely easy to implement, since it only modifies the starting vector for the Krylov method and it requires no additional computations at each iteration.

Among the Krylov methods that we tested, the one based on the EDS and the S\&I method with $\xi = -\sqrt{\abs{\lambda_2 \lambda_n}}$ converge quickly regardless of the parameters $\alpha$ and $t$; however, these methods require the computation of approximations to the eigenvalues $\lambda_2$ and $\lambda_n$ of the graph Laplacian. On the other hand, the S\&I method with $\xi = -t^{-2/\alpha}$ requires no previous knowledge of the spectrum of $L$, but its rate of convergence is more sensitive to changes in the parameters; even so, this method can sometimes outperform the others, especially when $t$ is large.

%~ \appendix
%~ \section{An example appendix} 
%~ \lipsum[71]

\section*{Acknowledgements}
We would like to thank Leonardo Robol and Valeria Simoncini for helpful suggestions.

%~ \section*{Acknowledgments}
%~ We would like to acknowledge the assistance of volunteers in putting
%~ together this example manuscript and supplement.

\FloatBarrier

\bibliographystyle{siamplain}
\bibliography{ms}
\end{document}